\documentclass[10pt]{amsart}

\usepackage[utf8]{inputenc}
\usepackage{multirow}
\usepackage{longtable}
\usepackage{amsfonts}
\usepackage{float}

\usepackage[backend=bibtex,style=alphabetic,sorting=anyt]{biblatex} 
\usepackage{amssymb,amscd,amsthm, verbatim,amsmath,color,fancyhdr, mathrsfs}
\usepackage{graphicx}
\usepackage{turnstile}
\usepackage{changepage}   

\usepackage{comment}

\makeatletter
\newcommand*{\da@rightarrow}{\mathchar"0\hexnumber@\symAMSa 4B }
\newcommand*{\da@leftarrow}{\mathchar"0\hexnumber@\symAMSa 4C }
\newcommand*{\xdashrightarrow}[2][]{%
  \mathrel{%
    \mathpalette{\da@xarrow{#1}{#2}{}\da@rightarrow{\,}{}}{}%
  }%
}
\newcommand{\xdashleftarrow}[2][]{%
  \mathrel{%
    \mathpalette{\da@xarrow{#1}{#2}\da@leftarrow{}{}{\,}}{}%
  }%
}
\newcommand*{\da@xarrow}[7]{%
  \sbox0{$\ifx#7\scriptstyle\scriptscriptstyle\else\scriptstyle\fi#5#1#6\m@th$}%
  \sbox2{$\ifx#7\scriptstyle\scriptscriptstyle\else\scriptstyle\fi#5#2#6\m@th$}%
  \sbox4{$#7\dabar@\m@th$}%
  \dimen@=\wd0 %
  \ifdim\wd2 >\dimen@
    \dimen@=\wd2 %
  \fi
  \count@=2 %
  \def\da@bars{\dabar@\dabar@}%
  \@whiledim\count@\wd4<\dimen@\do{%
    \advance\count@\@ne
    \expandafter\def\expandafter\da@bars\expandafter{%
      \da@bars
      \dabar@ 
    }%
  }%
  \mathrel{#3}%
  \mathrel{%
    \mathop{\da@bars}\limits
    \ifx\\#1\\%
    \else
      _{\copy0}%
    \fi
    \ifx\\#2\\%
    \else
      ^{\copy2}%
    \fi
  }%
  \mathrel{#4}%
}
\makeatother

\usepackage{tikz-cd} 
\usepackage{enumitem}

\newtheorem{thm}{Theorem}[section]
\newtheorem{prop}[thm]{Proposition}
\newtheorem{lem}[thm]{Lemma}
\newtheorem{cor}[thm]{Corollary}

\newtheorem{conj}[thm]{Conjecture}
\theoremstyle{definition}
\newtheorem{ex}[thm]{Example}
\newtheorem{defn}[thm]{Definition}
\newtheorem{alg}[thm]{Algorithm}
\theoremstyle{remark}
\newtheorem{rem}[thm]{Remark}

\title{Rational Contractions of Fiber Type on $\overline{M}_{0,6}$}

\author{Eric Jovinelly}

\date{\today}

\begin{document}
\maketitle
\addtocontents{toc}{\setcounter{tocdepth}{1}} 

\begin{abstract}
We identify a set of initial rational contractions of fiber type on $\overline{M}_{0,6}$.  
Our proof uses a new algorithm we develop for verifying descriptions of the cone of effective divisors on varieties without elementary rational contractions of fiber type.  On Mori dream spaces, our algorithm identifies faces of 
the cone of nef curves associated to initial rational contractions of fiber type.  
\end{abstract}

\section{Introduction}
Rational contractions of fiber type are an important class of rational maps closely related to the MMP.  These maps have been used to reduce questions about sources of larger dimension to statements about lower-dimensional targets.  For instance, \cite{Castravet_2015} \cite{gonzález2017balanced} \cite{hausen2016blowing} determined the moduli space $\overline{M}_{0,n}$ of stable rational curves with $n \geq 10$ markings is not a Mori dream space (MDS) by identifying rational contractions of fiber type $\overline{M}_{0,n} \dashrightarrow S$ to non-MDS surfaces $S$.  

$\overline{M}_{0,6}$ is a Mori dream space \cite{castravet2008cox}.  Rational contractions of fiber type on $\overline{M}_{0,6}$ therefore have MDS-codomain.  We show each factors rationally through one of the following \textit{initial} rational contractions of fiber type (see Definition \ref{def initial}).  These maps 
are described below in terms of the Kapranov basis on $\text{Pic}(\overline{M}_{0,6})$ given in Sec.\ 2.




\begin{thm}\label{rational contractions}
Every rational contraction of fiber type of $\overline{M}_{0,6}$ factors through one of the three following $S_6$-orbits of initial rational contractions of fiber type:
\begin{enumerate}
    \item A forgetful morphism $\pi : \overline{M}_{0,6} \rightarrow \overline{M}_{0,5}$,
    \item A rational map $\pi : \overline{M}_{0,6} \dashrightarrow Y$ factoring a product of iterated forgetful morphisms $\phi = \pi_{ij} \times \pi_{mn}: \overline{M}_{0,6}\rightarrow \overline{M}_{0,4} \times \overline{M}_{0,4}$ forgetting distinct markings,
    \item A rational map $\pi : \overline{M}_{0,6} \dashrightarrow Y$ factoring the rational map $\phi: \overline{M}_{0,6}\dashrightarrow \mathbb{P}^2$ induced by the moving part of either $|4H - 3E_i - \sum_{j,k\neq i} E_{jk}|$ or $|3H - E_{im} -E_{jm} -E_{km} -E_{in} -E_{jn} -E_{kn}|$, where $\{i,j,k,m,n\}=\{1,2,3,4,5\}$. 
\end{enumerate}
Here, $Y$ is isomorphic to a Keel-Vermiere divisor, realized in
both cases as the blow-up of $\phi$'s target along images of boundary divisors contracted to points by $\phi$.
\end{thm}

Our proof of Theorem \ref{rational contractions} does not use Castravet's description of the Cox ring $\text{Cox}(\overline{M}_{0,6})$ of $\overline{M}_{0,6}$ \cite{castravet2008cox}.  Incorporating her results leads to the following characterization of generators of $\text{Cox}(\overline{M}_{0,n})$ for $n = 5,6$.  



\begin{cor}\label{cox ring}
Suppose $n = 5,6$. A Cartier divisor $E$ is a generator of $\text{Cox}(\overline{M}_{0,n})$ iff $E\subset \overline{M}_{0,n}$ is a rigid section of a rational contraction of relative dimension one.

\end{cor}

Rational contractions of fiber type $\pi : X \dashrightarrow Y$ on a MDS are 
usually described by the subset $\pi^*\text{Nef}(Y)\subset \text{Mov}(X)$ of moving divisors.  We instead associate \textit{initial} rational contractions of fiber type to certain faces of the nef cone $\text{Nef}_1(X)$ of curves in Prop.\ \ref{rational contractions cor}.  General fibers of initial rational contractions are swept out by curves on any extreme ray of the associated face.  The process of identifying these faces (rays in Lemma \ref{lastLem}) provides the first computer-free proof of the following result.

\begin{thm}[\cite{Hassett_2002}]\label{coneThm}
Let $n = 5,6$ and $\mathcal{B}$ be the set of boundary divisors on $\overline{M}_{0,n}$.  Suppose $D\in N^1(\overline{M}_{0,n})$ satisfies $D|_B \in \text{Eff}(B)$ for all $B\in\mathcal{B}$.  Then $D\in\text{Eff}(\overline{M}_{0,n})$. 
\end{thm}

Like \cite{castravet2008cox}, our proof of Theorem \ref{coneThm} demonstrates by hand that $\text{Eff}(\overline{M}_{0,6})$ is generated by Keel-Vermiere and boundary divisors.  Translating this description of $\text{Eff}(\overline{M}_{0,6})$ to the statement of Theorem \ref{coneThm} previously required computer computation.  We develop an alternative algorithm in Section 4, Algorithm \ref{negCurvesAlg}, that allows for a computer-free analysis.  This algorithm leverages combinatorial patterns underlying the relationship between $\text{Mov}(X)$ and $\text{Eff}(X)$ in varieties $X$ with many rigid extremal divisors swept out by families of \textit{negative curves} (see Definition \ref{def negative curve}).  

When used to verify a description of $\text{Eff}(X)$ for some MDS $X$, Algorithm \ref{negCurvesAlg} identifies the faces of $\text{Nef}_1(X)$ associated to initial rational contractions of fiber type on $X$.  To recover the associated initial contraction, we apply results about \textit{special} rational contractions \cite{CasagrandeRational} (see Definition \ref{def special}) to rational contractions of fiber-type obtained from a MMP with scaling \cite{araujocone} \cite{lehmanncone} \cite{CasagrandeCone}.

Previous work \cite[Table~9]{M07arxivpaper} shows Theorem \ref{coneThm} is false for $n=7$.  This highlights a stark difference between $\text{Eff}(\overline{M}_{0,6})$ and $\text{Eff}(\overline{M}_{0,7})$: as any nonboundary extreme divisor is effective upon restriction to each boundary divisor, the former is as large as possible given natural inductive constraints, while the latter is not.  We conjecture this difference reflects a fundamental change to the rational contractions of fiber type on $\overline{M}_{0,n}$.  Data in \cite[Section~9]{M07arxivpaper} may be used to study such a difference using generalizations of methods in Section \ref{Section MDS}. 

\textbf{Outline:} Section \ref{Section2} defines \textit{negative curves} and reviews Kapranov's construction of $\overline{M}_{0,n}$ as an iterated blow-up of $\mathbb{P}^{n-3}$.  In Section \ref{Section MDS}, we define \textit{initial} rational contractions and describe their relation to nef curves.  Section \ref{NegTech} describes general techniques for working with cones defined by negative curve classes, which are applied in the proofs of Theorems \ref{coneThm} and \ref{rational contractions} in Sections \ref{Section10} and \ref{Section9}.

\textbf{Acknowledgments:}  The author thanks Eric Riedl, Brian Lehmann, Izzet Coskun, and Ana-Maria Castravet, and David Jensen for helpful feedback.



\tableofcontents

\section{Background and Notation}\label{Section2}
\subsection{Background}  Let $X$ be a variety and $\text{Eff}(X)$ be the full-dimensional, pointed cone of effective divisors inside $N^1(X)$, with closure $\overline{\text{Eff}}(X)$.  Its dual, $\text{Nef}_1(X)\subset N_1(X)$ is the cone of nef (free) curves on $X$ \cite{bdpp}.  
Extreme rays of $\overline{\text{Eff}}(X)$ often contain the class of a rigid, irreducible divisor $D$ swept out by a family of integral curves of class $c$ satisfying $c.D < 0$.  Such a class $c$ provides a hyperplane in $N^1(X)$ that separates $D$ from every other extreme ray of $\overline{\text{Eff}}(X)$.  We call such curves \textit{negative curves} (on/sweeping out $D$).

\begin{defn}\label{def negative curve}
Let $D\subset X$ be a Cartier divisor.  A curve class $c\in N_1(X)$ is called a \textit{negative curve} (on/sweeping out $D$) if $c.D < 0$ and there exists a family of integral curves, numerically equivalent to $c$, whose image dominates $D$.  
\end{defn}

Negative curves often describe facets of the movable cone $\text{Mov}(X)$ of divisors \cite{lehmann_movable} \cite{Choi}.  This extends classical results proving duality between the closed cone $\overline{\text{Nef}}(X)$ of nef divisors and the Mori cone $\overline{\text{NE}}(X)$ of effective curves. We have
$$\text{Nef}(X) \subseteq \text{Mov}(X) \subseteq \text{Eff}(X).$$
When $X$ is a Mori dream space, all three of these are closed and polyhedral.



\subsection{Notation}\label{Notation}  Throughout our paper, $\mathcal{D}$ will be a finite set of extreme divisors on a variety $X$ and $\mathcal{C}$ will be a finite collection of negative curves sweeping out divisors in $\mathcal{D}$.  We let $\mathfrak{E}\subset N^1(X)$ denote the cone generated $\mathcal{D}$, and define the cone 
$$\mathfrak{M}=\{ D\in N^1(X) | D.c \geq 0 \text{ for all } c\in \mathcal{C}\}.$$



We focus primarily on $X=\overline{M}_{0,6}$.  Kapranov's construction realizes $\overline{M}_{0,6}$ as the iterated blow up of $\mathbb{P}^3$ along 5 linearly general points $p_i$, $1\leq i \leq 5$, and the strict transforms of the 10 lines $l_{ij} = \overline{p_i p_j}$ they span.  We let $H$ denote the pullback of the hyperplane class from $\mathbb{P}^3$ under the iterated blow-up $\overline{M}_{0,6}\rightarrow \mathbb{P}^3$, $E_i$ denote the exceptional divisor over $p_i$, and $E_{ij}$ denote the exceptional divisor over the strict transform of $l_{ij}$.  These comprise a basis for $N^1(\overline{M}_{0,6})\cong\text{Pic}(\overline{M}_{0,6}) \otimes_\mathbb{Z} \mathbb{R}$.  

A corresponding basis for $N_1(\overline{M}_{0,6})$ consists of $l, e_i, \text{ and } e_{ij}$, where $l$ is the strict transform of a general line under $\overline{M}_{0,6}\rightarrow \mathbb{P}^3$, $e_i$ is the strict transform of a general line under the blow-up $E_i\rightarrow \mathbb{P}^2$, and $e_{ij}$ is the fiber over a general point in the restriction of Kapranov's map to $E_{ij}\rightarrow \mathbb{P}^1\subset\mathbb{P}^3$.  Note $l.H =1$, $e_i .E_i = -1$, and $e_{ij} .E_{ij} = -1$.  All other pairing of basis elements are $0$.

\subsection{Mori Dream Spaces} Properties of Mori dream spaces (MDS's), defined below, are well documented.  We review those related to rational contractions.

\begin{defn}[\cite{MDSpaper}]
A Mori dream space (MDS) is a normal, projective $\mathbb{Q}$-factorial variety $X$ such that
\begin{enumerate}
    \item $\text{Pic}(X)$ finitely generated and $\text{Pic}(X)_\mathbb{Q} \cong N_1(X)_\mathbb{Q}$;
    \item $\text{Nef}(X)$ is generated by finitely many semi-ample line bundles;
    \item There are finitely many small $\mathbb{Q}$-factorial modifications (SQMs) $f_i : X \dashrightarrow X_i$ such that each $X_i$ satisfies (1) and (2), and $\text{Mov}(X) = \cup_i f_i^*(\text{Nef}(X_i))$.
\end{enumerate}
\end{defn}
\noindent Rational contractions of Mori dream spaces admit the following description \cite{MDSpaper}.

\begin{defn}
A rational contraction $\pi : X \dashrightarrow Y$ of a MDS is the composition of a SQM $f : X \dashrightarrow \tilde{X}$ and a contraction $\tilde{\pi} : \tilde{X} \twoheadrightarrow Y$ (map with connnected fibers).
\end{defn}

Rational contractions of a MDS $X$ are in bijection with the faces of a certain decomposition of $\text{Mov}(X)$ into a fan.  
Moreover, the MMP can be carried out for every divisor on $X$ \cite[Proposition~1.11(1)]{MDSpaper}, and every rational contraction can be factored as a finite sequence of flips followed by small, divisorial, and fiber type elementary contractions.  If $f: X \dashrightarrow Y$ is a rational contraction and $Y$ is $\mathbb{Q}$-factorial, then $Y$ is a MDS.  In other words, when the above factorization contains no small contractions, the target is also a MDS.  

A rational contraction $\pi:X\dashrightarrow Y$ is of fiber type iff $\pi^*\text{Nef}(Y)$ is contained in the boundary of $\text{Eff}(X)$.  By factoring through a SQM, we may define $\pi_* : N_1(X)\rightarrow N_1(Y)$ as the dual of $\pi^*$.  It follows that $\pi$ is of fiber type iff $\ker \pi_*$ intersects $\text{Nef}_1(X)$.  For our purposes, we highlight one property about rational contractions. 


\begin{lem}
Let $X$ be a MDS and $v\in \text{Nef}_1(X)$ be an extreme ray.  Suppose there exist negative curves $c_i$ and an expression $v = \sum a_i c_i$ with $a_i > 0$.  For any rational contraction $\pi : X \dashrightarrow Y$, if $\pi_* v = 0$, then $\pi_* c_i = 0$ for all $i$.
\end{lem}

\begin{proof}
Since $v$ and $c_i$ admit families of integral curves sweeping out subsets of codimension at most 1, we may resolve the indeterminacies of $X \dashrightarrow Y$ while preserving effectivity of $v$ and each $c_i$.  Our result then follows from standard lemmas.
\end{proof}

\begin{cor}\label{extreme rays cor}
Let $X$ be a MDS and $v,w \in \text{Nef}_1(X)$.  Suppose there exists negative curves $c_i$ and expressions $v = \sum a_i c_i$ and $w = \sum b_i c_i$ with $a_i >0$, $b_i \geq 0$.  For any rational contraction $\pi: X \dashrightarrow Y$, if $\pi_* v = 0$ then $\pi_* w = 0$ as well.
\end{cor}
\begin{proof}
Let $f: \tilde{X} \rightarrow X$ and $\tilde{\pi} = \pi \circ f : \tilde{X} \rightarrow Y$ be a resolution of indeterminacies.  Since representatives of each $c_i$ sweep out a divisor, the pullback $f^*(c_i)$ of each class lies in $\overline{NE}(\tilde{X})$.  Therefore, as $\tilde{\pi}_*(f^* v) = 0$, $\tilde{\pi}_*(f^*c_i) = 0$ for all $i$.  Hence $\tilde{\pi}_*(f^* w) = \pi_* w = 0$.
\end{proof}

\section{Rational Contractions of Fiber Type on Mori Dream Spaces}\label{Section MDS}

Throughout this section, let $X$ be a MDS.  We study rational contractions on $X$ through a finite set of \textit{initial} rational contractions of fiber type (Definition \ref{def initial}).  We associate these \textit{initial} rational contractions of fiber type to certain extreme rays of $\text{Nef}_1(X)$ in Proposition \ref{rational contractions cor}.  We translate this association to a statement about \textit{negative curves} in Corollary \ref{cor 2}.  This allows us to describe methods of identifying initial rational contractions of fiber type in Section \ref{identifying initial contractions}.  


\begin{defn}\label{def initial}

Two rational contractions $\pi_i : X \dashrightarrow Y_i$ are \textit{equivalent} if both $Y_i$ are $\mathbb{Q}$-factorial and $\pi_i$ differ by an SQM of their targets.  
A set $S$ of equivalence classes of rational contractions $\pi_s : X \dashrightarrow Z_s$ of fiber-type 
is \textit{initial} if 
\begin{enumerate}
    \item For any rational contraction $\phi : X \dashrightarrow Y$ of fiber type, there exists a factorization $f \circ \pi_s$ of $\phi$ away from a codimension 2 subset of $X$ such that $f : Z_s \dashrightarrow Y$ is a rational contraction.
    \item If in (1), $\phi = \pi_{s'}$ for some $s' \in S$, then $s' = s$
    ($f$ is a SQM).
\end{enumerate}
\end{defn}

\begin{rem}
$S$ is finite and each $Z_s$ must be $\mathbb{Q}$-factorial.  Thus, \textit{initial} rational contractions are equivalence classes of maps with targets defined up to SQM. 
\end{rem}


\subsection{Relation to Nef Curves} 

Initial rational contractions are closely related to \textit{special} rational contractions (Definition \ref{def special}).  We make this relation precise in Proposition \ref{rational contractions cor} using literature on nef curves deriving from \cite{batyrevcone}.

\begin{defn}[Defnition 2.3, Remark 2.4 \cite{CasagrandeRational}]\label{def special}
Let $f : X \dashrightarrow Y$ be a rational contraction of fiber type.  We say $f$ is \textit{special} if $Y$ is $\mathbb{Q}$-factorial and for every prime divisor $D \subset X$, $\text{codim } f(D) \leq 1$.
\end{defn}

\begin{prop}[Proposition 2.13 \cite{CasagrandeRational}]\label{special contraction}
Let $f: X \dashrightarrow Y$ a rational contraction of fiber type.  Then $f$ can be factored as
$$ X \xdashrightarrow{g} Z \xrightarrow{h} Y$$
where $g$ is a special rational contraction and $h$ is birational.  Moreover, such a factorization is unique up to composition with a SQM of $Z$.
\end{prop}

Moreover, \cite[Construction~2.11]{CasagrandeRational} may be used to explicitly construct this factorization using a sequence of relative MMPs for $-D$, where $D \subset X$ is a prime divisor with $\text{codim } f(D) > 1$ (more accurately such that $\overline{f(D)}$ is not a $\mathbb{Q}$-Cartier divisor in $Y$).

It is immediate that a set $S$ of initial rational contractions exists, and that each must be a special rational contraction.  However, not every special rational contraction is initial.  Indeed, the projection $\mathbb{P}^1 \times \mathbb{P}^1 \times \mathbb{P}^1 \rightarrow \mathbb{P}^1$ is special but not initial.  Proposition \ref{rational contractions cor} relates initial rational contractions of fiber-type to extreme rays of $\text{Nef}_1(X)$ using Proposition \ref{extreme contraction}, which generalizes contractions of \textit{coextremal rays} defined in \cite{batyrevcone}.

\begin{prop}[5.1-2 \cite{araujocone}, 2.17 \cite{CasagrandeCone}, 1.4 \cite{lehmanncone}]\label{extreme contraction}  For every extreme ray $v \in \text{Nef}_1(X)$, 
there is a birational contraction $f: X \dashrightarrow X'$ to a MDS $X'$ such that $f_* v \in \overline{NE}(X')$ is extreme.  Moreover the contraction $g : X' \rightarrow Y$ associated to $f_* v$ defines a rational contraction $\pi_v := g \circ f$ of fiber type whose resolution $h : \tilde{X} \rightarrow X$, $\tilde{\pi}_v : \tilde{X} \rightarrow Y$ is determined up to birational equivalence by the following properties: 
\begin{itemize}
    \item 
    For every movable curve $C \subset \tilde{X}$ with $h_*[C]$ proportional to $v$, $\tilde{\pi}_{v*}[C] = 0$.
    \item There is a movable curve $C$ through a general pair of points in a general fiber of $\tilde{\pi}_v$ with $h_*[C]$ proportional to $v$.
\end{itemize}
\end{prop}

\begin{prop}\label{rational contractions cor}
Let $\text{Mov}(X)^* \subset N_1(X)$ be dual to the cone of moving divisors.  Initial rational contractions of fiber-type on $X$ are in bijection with minimal faces of $\text{Mov}(X)^*$ that intersect $\text{Nef}_1(X)$.  For each such minimal face $F\subset \text{Mov}(X)^*$, the associated initial rational contraction is the unique special rational contraction $\pi_v$ associated by \ref{extreme contraction} to any extreme ray $v \in \text{Nef}_1(X)$ lying in $F$.
\end{prop}
\begin{proof}
Using Proposition \ref{special contraction}, we may assume $\pi_v$ from \ref{extreme contraction} is special.  Such a map $\pi_v$ must be unique up to SQM of the target.  It follows that a set of initial rational contractions is comprised of such maps $\pi_v$.  Moreover, if $v,w \in \text{Nef}_1(X)$ are two extreme rays and the minimal face of the dual of $\text{Mov}(X)$ containing $w$ also contains $v$, then by Corollary \ref{extreme rays cor} 
$\pi_w$ must factor rationally through $\pi_v$.
\end{proof}

\begin{ex}
\textbf{Caution:} Just as not every special contraction is initial, the special rational contraction $\pi_v$ may not be initial, and for two extreme rays $v \neq w$, it is possible that $\pi_v = \pi_w$.  Indeed, consider the Fano blow-up $X \rightarrow \mathbb{P}^3$ with center a curve $C$ of degree 6 and genus 3 \cite{mori1983classification}.  Extreme rays of $\text{Nef}_1(X)$ are spanned by the class of a general line in $\mathbb{P}^3$ and the strict transform of a cubic meeting $C$ transversely at 8 points \cite{BurkeJovinelly}.  The only rational contraction of fiber-type is the trivial map $X \rightarrow \text{Spec}(k)$.  This may be interpreted in two ways:
\begin{enumerate}
    \item The cone of moving divisors $\text{Mov}(X)$ on $X$ has trivial intersection with the boundary of $\text{Eff}(X)$,
    \item The minimal face of $\overline{NE}(X)$ containing either extreme ray of $\text{Nef}_1(X)$ is the same, and $\text{Mov}(X) = \text{Nef}(X)$.
\end{enumerate}
Were $\text{Mov}(X) \neq \text{Nef}(X)$,
the minimal face of $\overline{NE}(\tilde{X})$ for a SQM $X \rightarrow \tilde{X}$ containing one extreme ray of $\text{Nef}_1(X)$ may shrink in dimension.  
\end{ex}

In lieu of computing $\text{Mov}(X)$ to apply Proposition \ref{rational contractions cor}, one may use properties of negative curves.  
For instance, when $X = \overline{M}_{0,6}$, each extreme ray of $\text{Nef}_1(X)$ may be expressed as a sum of negative curves sweeping out boundary divisors.  This allows us to apply the following corollary.

\begin{cor}\label{cor 2}
Suppose $\mathcal{C}$ is a finite collection of negative curves that bounds a cone $\mathfrak{M} \subseteq \text{Eff}(X)$.  Let $S$ be a minimal set of extreme rays of $\text{Nef}_1(X)$ such that every face of $\mathfrak{M}^*$ meeting $\text{Nef}_1(X)$ contains a ray in $S$.  Initial rational contractions of fiber-type on $X$ are the unique special rational contractions $\pi_v$ associated to $v \in S$.
\end{cor}
\begin{proof}
Clearly $\text{Mov}(X) \subseteq \mathfrak{M}$, so we may replace $\text{Mov}(X)^*$ in Proposition \ref{rational contractions cor} with the cone $\mathfrak{M}^*$ generated by $\mathcal{C}$.  Since $\mathfrak{M} \subseteq \text{Eff}(X)$, $\mathfrak{M}^* \supseteq \text{Nef}_1(X)$, and Corollary \ref{extreme rays cor} implies our claim.
\end{proof}

\begin{rem}
For any special rational contraction $X \dashrightarrow Y$, there are finitely many prime divisors in the base $Y$ dominated by reducible divisors in $X$ \cite{CasagrandeRational}.  Their preimages are unions of divisors swept out by \textit{negative curves} which sum to the class of a nef curve on a general fiber.
\end{rem}


To identify a set $S$ of extreme rays as above, we use Algorithm \ref{negCurvesAlg}.  Given finite collections $\mathcal{C}$ of \textit{negative curves} and $\mathcal{D}$ of effective divisors on a variety $X$, Algorithm \ref{negCurvesAlg} provides a method for verifying whether the cone $\mathfrak{M}$ bounded by $\mathcal{C}$ is contained in $\text{Eff}(X)$.  When it is, Algorithm \ref{negCurvesAlg} shows $\mathcal{D}$ generates $\text{Eff}(X)$ and identifies a set $S$ of extreme nef curves as in Corollary \ref{cor 2}.  The following paragraphs describe how the special rational contractions $\pi_v$ may be identified.

\subsection{Identifying Initial Rational Contractions} \label{identifying initial contractions}

Proposition \ref{extreme contraction} describes how a rational contraction of fiber type $\phi = \pi_v : X \dashrightarrow Y$ may be obtained from any extreme ray $v \in \text{Nef}_1(X)$.  Corollary \ref{cor 2} may be used to identify a set $S$ of these extreme rays for which the corresponding special $\pi_v$ are initial.  If $f : X \dashrightarrow \tilde{X}$ is an SQM such that the dimension of smallest face $F\subseteq \overline{\text{NE}}(\tilde{X})$ containing $v \in \text{Nef}_1(X) = \text{Nef}_1(\tilde{X})$ is minimal among all SQMs, then a special $\pi_v$ is given by $\pi \circ f$, where $\pi : \tilde{X} \rightarrow \tilde{Y}$ is the contraction associated to $F$.  

\cite{CasagrandeRational} offers a geometric perspective, on which we elaborate when $X = \overline{M}_{0,6}$.  This allows us to resolve a birationally equivalent rational contraction $\phi : X \dashrightarrow Y$ with $\mathbb{Q}$-factorial target to obtain a special rational contraction $\pi_v : X \dashrightarrow \tilde{Y}$, where $\tilde{Y} \rightarrow Y$ is a blow-up.  In each case, the rational contraction $\phi : X = \overline{M}_{0,6} \dashrightarrow Y$ is identified in \cite[Section~5.2]{Hassett_2002} as a contraction of the \textit{coextremal ray} $v$ \cite{batyrevcone}.

In all cases, resolving contractions $\phi : \overline{M}_{0,6} \dashrightarrow Y$ of \textit{coextremal rays} in \cite[Section~5.2]{Hassett_2002} as described above requires flipping 
rigid one-dimension components (flipping curves) of special, two-dimensional fibers of $\phi$.  These flipping curves $f_i \subset \overline{M}_{0,6}$ meet each member of a family of negative curves on $\overline{M}_{0,6}$ sweeping out a boundary divisor $E$ (a two-dimensional component of the special fiber of $\phi$).  If the negative curves have class $e$, then $e + \sum f_i$ is the class of a general fiber of $\phi$.  The Attiyah-flop of each $f_i \in \overline{NE}(X)$ produces a SQM $\tilde{X}$ and a rational contraction $\tilde{\phi} : \tilde{X} \dashrightarrow \tilde{Y}$ factoring $\phi$ into $\psi \circ \tilde{\phi}$, where $\psi : \tilde{Y} \rightarrow Y$ is the blow-up of the point $p = \phi(f_i)$.  The strict transform of $E$ in $\tilde{X}$ dominates the exceptional locus $\psi^{-1}(p)$.  After performing this operation for each divisor $E \subset \overline{M}_{0,6}$ whose image under $\phi$ has codimension $\geq 2$, we obtain a 
special rational contraction $\pi : \overline{M}_{0,6} \dashrightarrow \tilde{X} \rightarrow \tilde{Y}$. 

\section{A General Approach To Cones Defined by Negative Curves}\label{NegTech}
To prove Theorem \ref{coneThm}, \cite{Hassett_2002} uses a computer algorithm to compute the dual description of a cone $\mathfrak{E}\subset \text{Eff}(\overline{M}_{0,n})$ 
and proves each facet is a positive linear combination of negative curves on divisors contained in $\mathfrak{E}$.  The first step of this procedure, computing a dual description, is too complicated to verify by hand and sometimes impossible with modern technology.  
This section describes an alternative algorithm, Algorithm \ref{negCurvesAlg}, that verifies $\mathfrak{E}=\text{Eff}(X)$ without computing a dual description.  We use Algorithm \ref{negCurvesAlg} in Section \ref{Section10} to provide a computer-free proof of Theorem \ref{coneThm}.  
Like \cite{castravet2008cox}, our proof demonstrates by hand that $\text{Eff}(\overline{M}_{0,6})$ is generated by boundary and Keel-Vermiere divisors; however, we also recover information about rational contractions on $\overline{M}_{0,6}$ using Corollary \ref{cor 2}.

As described in Section \ref{Notation}, let $X$ be a variety, $\mathcal{D}$ be a finite set of extreme divisors on $X$, and $\mathcal{C}$ be a finite collection of negative curves sweeping out divisors in $\mathcal{D}$.  We let $\mathfrak{E}$ be the cone generated by $\mathcal{D}$, and
$$\mathfrak{M}=\{ D\in N^1(X) | D.c \geq 0 \text{ for all } c\in \mathcal{C}\}.$$
To show $\mathfrak{E} = \text{Eff}(X)$, it is sufficient to show $\mathfrak{M} \subset \mathfrak{E}$, as $\text{Eff}(X)\subseteq \mathfrak{E}+\mathfrak{M}$.  We use the fact that $\mathfrak{M}$ is bounded by negative curves sweeping out divisors in $\mathfrak{E}$ to reduce showing $\mathfrak{M} \subset \mathfrak{E}$ to simpler problems about cones of much lower dimension.  
More specifically, it is sufficient to check that $\mathfrak{M}\cap \{ c=0\}\subset\mathfrak{E}$ for finitely many nef curve classes $c$ generated by $\mathcal{C}$.



\subsection{Notation}  Since our subject of interest is cones, we say that a set of vectors \textit{generates} another vector $v$ if $v$ lies in their nonnegative linear span.  In particular, a set of curves $I$ \textit{generates} another curve $c$ if $c$ can be written as a nonnegative linear combination of curves in $I$.  In what follows, let $I$ be a subset of $\mathcal{C}$ and $F$ be a face of $\mathfrak{M}$.  We make the following definitions.
\begin{defn}\label{def nef minimal}
A subset $I$ of $\mathcal{C}$ is called \textit{nef-minimal} if $I$ generates a nonzero nef curve, but no proper subset of $I$ generates a nonzero nef curve.
\end{defn}
\begin{rem}\label{matrixAlg}
Suppose curves in $I=\{c_1,\dots, c_n\}$ sweep out distinct divisors $D_i=\mathbf{N}(c_i)$.  Then $I$ generates a nonzero nef curve if and only if placing the $n\times n$ matrix $A_I=[(D_i.c_j)]_{i,j\leq n}$ in row elechelon form by adding positive multiples of rows to other rows produces a row with nonnegative entries.  If we only add positive multiples of each row to rows beneath it, this is equivalent to the appearance of a nonnegative diagonal entry.
\end{rem}
\begin{defn}
For $I\subset\mathcal{C}$, let $\mathbf{F}(I)\subset \mathfrak{M}$ be the face of $\mathfrak{M}$ defined by the vanishing of all curves in $I$. 
Similarly, for a face $F$ of $\mathfrak{M}$, let $\mathbf{I}(F)\subset\mathcal{C}$ be those curves which vanish on $F$.
\end{defn}
Through abuse of notation, if $c$ is a nef curve generated by $\mathcal{C}$, we define $\mathbf{F}(c)$ as the intersection of $\mathfrak{M}$ with the vanishing of $c$.  This will be a face of $\mathfrak{M}$.  
\begin{rem}
A curve $c\in\mathcal{C}$ is in $\mathbf{I}(\mathbf{F}(I))$ if and only if $-c$ is generated by $\mathcal{C}\cup\{-\alpha | \alpha\in I\}$.
\end{rem}
\begin{rem}\label{Relation faces and subsets}
The relationship between faces of $\mathfrak{M}$, subsets of $\mathcal{C}$, and nef curves is essential to Algorithm \ref{negCurvesAlg} and the proof of Theorem \ref{coneThm}.  Suppose $I\subset\mathbf{I}(F)$ and $c=\sum_{c_i\in I} a_i c_i$ is a nef curve with $a_i > 0$ for all $i$.  For any $D\in\mathfrak{M}$, if $D.c = 0$, then $D.c_i = 0$ for all $i$, and consequently $D\in\mathbf{F}(I)$.  
It follows that $\mathbf{F}(I)=\mathbf{F}(c)$.  If $I=\mathbf{I}(F)$, then $F=\mathbf{F}(I)=\mathbf{F}(c)$ as well, but in general we only have $F\subseteq\mathbf{F}(c)$ for any nef curve $c$ generated by $\mathbf{I}(F)$.
\end{rem}
\begin{defn}\label{def cover}
We say a nef curve $c$ (generated by $\mathcal{C}$) \textit{covers} a nef-minimal subset $I\subset\mathcal{C}$ if $\mathbf{F}(c)\supseteq \mathbf{F}(I)$.  This happens if and only if $\mathbf{I}(\mathbf{F}(I))$ generates $c$.
\end{defn}
We now state our method as a proposition.
\begin{prop}\label{nefCovers}
Let $X,\mathcal{D},\mathcal{C},\mathfrak{E},\text{ and } \mathfrak{M}$ be as defined.  To show $\mathfrak{M}\subset\mathfrak{E}$, and thus that $\mathfrak{E}=\text{Eff}(X)$, it is sufficient to show that $\mathbf{F}(c)\subset\mathfrak{E}$ for a collection of nef curves generated by $\mathcal{C}$ that covers all nef-minimal subsets $I\subset\mathcal{C}$.
\end{prop}
\begin{proof}
Let $(\mathfrak{E}+\mathfrak{M})^*\subset N_1(X)$ be the dual of the convex hull of $\mathfrak{E}\cup\mathfrak{M}$.  By the definitions of $\mathfrak{E}$ and $\mathfrak{M}$, we see that $(\mathfrak{E}+\mathfrak{M})^*$ has the following explicit description:
\begin{align*}
(\mathfrak{E}+\mathfrak{M})^*=\big\{  c\in N_1(X) | & D.c \geq 0 \text{ for all } D\in\mathcal{D} \text{ and s.t. }\\
& c=\sum a_i c_i \text{ for some } c_i\in \mathcal{C} \text{ and } a_i\geq 0\big\}.
\end{align*}
By duality, $\mathfrak{E}^*\supseteq (\mathfrak{E}+\mathfrak{M})^*$.  Showing $\mathfrak{M}\subset\mathfrak{E}$ is equivalent to showing $\mathfrak{E}^*\subseteq (\mathfrak{E}+\mathfrak{M})^*$.  Were $\mathfrak{M}\nsubseteq\mathfrak{E}$, then some extreme ray $v$ of $\mathfrak{M}$, not contained in $\mathfrak{E}$, would need to define a facet of $(\mathfrak{E}+\mathfrak{M})^*$, and therefore pair to $0$ with some nonzero element of $(\mathfrak{E}+\mathfrak{M})^*$.  Note that every curve in $(\mathfrak{E}+\mathfrak{M})^*$ is a nef curve generated by $\mathcal{C}$.

Suppose $v\in \mathfrak{M}$ and $c\in (\mathfrak{E}+\mathfrak{M})^*$ satisfy $v.c=0$.  Since $c$ is a nef curve generated by $\mathcal{C}$, $c=\sum_{c_i\in I} a_i c_i$ (with $a_i > 0$) for some $I\subseteq \mathcal{C}$.  As in Remark \ref{Relation faces and subsets}, note that $v\in\mathfrak{M}$ and $v.c = 0$ imply $v.c_i = 0$ for all $i$, and consequently $v\in\mathbf{F}(I)$.  Since $I$ generates a nef curve, some subset $I'\subseteq I$ is nef-minimal. As $\mathbf{F}:\{\text{subset of }\mathcal{C}\}\rightarrow \{\text{faces of }\mathfrak{M}\}$ reverses inclusion, $v\in \mathbf{F}(I')$, which proves the our claim.
\end{proof}
This motivates the study of nef-minimal subsets of $\mathcal{C}$.  Remark \ref{matrixAlg} gives an explicit description of when $I\subset\mathcal{C}$ generates a nef curve, but checking all subsets of $\mathcal{C}$ is exponentially tedious.  To mitigate this, Lemma \ref{propertiesNefMinimal} remarks on the structure of nef-minimal subsets, and Lemma \ref{eliminate} describes criteria that narrow our search for nef-minimal subsets.  First, we need a few more definitions.

\begin{defn}
For $c\in \mathcal{C}$, let $\mathbf{N}(c)\in \mathcal{D}$ be the divisor swept out by $c$ and $\mathbf{P}(c)=\{ D\in \mathcal{D} | c.D > 0\}$.  For a divisor $D\in \mathcal{D}$, define $\mathbf{N}(D)=\{ c\in \mathcal{C} | c.D < 0\}$ and $\mathbf{P}(D)=\{ c\in \mathcal{C} | c.D > 0\}$.  Extend $\mathbf{N}(\cdot)$ and $\mathbf{P}(\cdot)$ to subsets of $\mathcal{C}$ and $\mathcal{D}$ by taking unions.
\end{defn}


\begin{defn}\textit{The Directed Graph $\mathbf{G}(I)$ associated to $I\subset\mathcal{C}$}:
$\mathbf{G}(I)$ has vertex set $\mathbf{N}(I)$ and a directed edge from $D$ to $D'$ if $\mathbf{N}(D)\cap \mathbf{P}(D') \cap I \neq \emptyset $.
\end{defn} 
\begin{figure}[H]\label{image directed graph}
\includegraphics[width=0.32\linewidth]{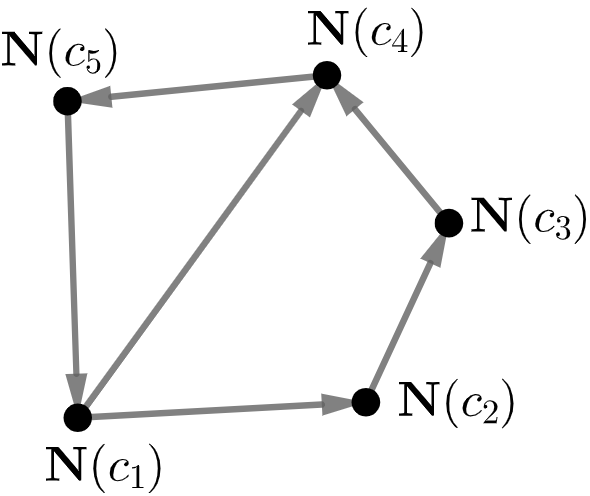}
\caption{An example of $\mathbf{G}(I)$.}
\end{figure}


These graphs are useful visual tools (see Figure \ref{image directed graph}) in the study of nef-minimal subsets $I\subset\mathcal{C}$, particularly when $\mathcal{C}$ is unimodal, as defined below.
\begin{defn}
A collection of negative curves $\mathcal{C}$ on a variety $X$ is \textit{unimodal} if, for every curve $c\in\mathcal{C}$ and every divisor $D\in\mathbf{N}(\mathcal{C})\setminus\{\mathbf{N}(c)\}$, either $D.c = 0$ or $D.c = -\mathbf{N}(c).c$.
\end{defn}

\begin{lem}\label{propertiesNefMinimal}
Every nef-minimal $I\subset\mathcal{C}$ satisfies the following properties:
\begin{enumerate}
\item No two curves in $I$ sweep out the same divisor.  In other words, if $a,b\in I$, then $\mathbf{N}(a)\neq \mathbf{N}(b)$. 
\item Every vertex $v\in \mathbf{G}(I)$ is the source and target of a directed edge, i.e. $I\subseteq \mathbf{N}(\mathbf{P}(I))\cap \mathbf{P}(\mathbf{N}(I))$.  If a proper subset $V$ of the vertices in $\mathbf{G}(I)$ are connected by a directed cycle, then the quotient of $\mathbf{G}(I)$ identifying all vertices in $V$ also satisfies this property.
\item For any two vertices $v_1, v_2 \in \mathbf{G}(I)$, there is a directed path from $v_1$ to $v_2$.  If $\mathcal{C}$ is unimodal, $\mathbf{G}(I)$ contains exactly one directed cycle, which must be Hamiltonian (this implies $I$ generates a nef curve).
\end{enumerate}
\end{lem}
\begin{proof}
Property (1): if $I$ is nef-minimal, then for any $D\in\mathbf{N}(I)$, $I$ generates a nef curve $c$ which pairs to 0 with all $D'\in\mathbf{N}(I)\setminus\{D\}$.  So, if two curves $c_1,c_2\in I$ sweep out $D$, we may split up $c=\sum a_i c_i$ into $c_x= a_1 c_1 + \sum_{i\geq 3} x_i c_i$ and $c_y= a_2 c_2 + \sum_{i\geq 3} y_i c_i$, with $x_i + y_i = a_i$, such that $c_x$ and $c_y$ also pair to $0$ with all $D'\in\mathbf{N}(I)\setminus\{D\}$.  At least one of $c_x$ and $c_y$ need to be nef.

Properties (2) and (3) are straightforward to verify.
\end{proof}

Let $S$ be a collection of nef curves generated by $\mathcal{C}$.  To check whether $S$ covers all nef-minimal subsets of $\mathcal{C}$, it is not necesary to find and check all nef-minimal $I\subset\mathcal{C}$.  We introduce the notion of \textit{elimination relative} $S$ for $I\subset \mathcal{C}$ to formalize which nef-minimal subsets we may skip finding.

\begin{defn}
Let $S$ be as above.  We say that $I\subset\mathcal{C}$ may be \textit{eliminated} (relative $S$) if we can show that for any nef-minimal $I'\supseteq I$, either $I'$ is covered by $S$, or there exists a nef-generating $I''\nsupseteq I$ with $\mathbf{F}(I'')\supseteq\mathbf{F}(I')$. 
We say that every nef-generating subset of $I''$ \textit{eliminates} every set containing $I'$.  If $\mathbf{F}(I'') = \mathbf{F}(I')$, we require that $I''$ was not previously eliminated by $I'$.
\end{defn}

\begin{lem}\label{eliminate}
 The following criteria allow us to eliminate $I$ relative $S$:
\begin{enumerate}

\item If 2 curves in $I$ sweep out the same divisor, 
or $\mathbf{F}(I)\subset \mathbf{F}(c)$ for some $c\in S$, we may eliminate $I$.


%
%
\item  Fix $c\in I$ 
and $\alpha$ generated by $I\setminus\{c\}$ s.t. 
$\mathbf{N}(c').(c+\alpha) \geq 0$ for all $c'\in I\setminus \{c\}$.  
 Suppose $I_0\subset\mathcal{C}\setminus\{c\}$ also generates $c+\alpha$.  
Let $G$ be a directed graph with vertex set $\mathcal{C}$.  For each previous subset $I' \ni c'$ eliminated by $I_0 '$ using this criteria ($c'$ corresponding to $c\in I$), add edges to $G$ from $c'$ to all curves in $I_0 '\setminus I'$.  We may eliminate $I$ if adding edges from $c$ to all curves in $I_0\setminus I$ does not create a directed cycle in $G$.
%
%
\item Let $D\in \mathbf{N}(\mathbf{I}(\mathbf{F}(I)))\setminus\mathbf{N}(I)$.  If for every choice of $c\in\mathbf{N}(D)$ the set $I\cup \{c\}$ has been eliminated (not by criteria (2) applied to $c\in I\cup\{c\}$ and $c+\alpha$), we may eliminate $I$ as well.\footnote{If, for every choice of $c=c_i\in \mathbf{N}(D)\cap \mathbf{I}(\mathbf{F}(I))$, $I\cup\{c_i\}$ is eliminated by criteria (2) with $c'=c_i'\in I$, then we also require the existence of some $c_i$ such that the addition of an edge from $c_i'$ to $c_i$ in the directed graph $G$ above does not create any directed cycles.}   
\end{enumerate}
\end{lem}
\begin{proof} 

Criteria (1) follows from Lemma \ref{propertiesNefMinimal} (1).  

Criteria (2): Let $I$, $I_0$, and $c+\alpha$ be as described.  If $I$ contains curves sweeping out the same divisor, by Lemma \ref{propertiesNefMinimal}(1) there are no nef-minimal $I\supset I$.  Likewise, if $I\setminus\{c\}$ generates a nef curve, no nef-minimal $I'\supset I$ exist.  Hence, we may replace $I_0$ with $I_0\cup I\setminus \{c\}$ and modify $\alpha$ by positive multiples of $c'\in I\setminus \{c\}$ until $\mathbf{N}(c').(c+\alpha) = 0$ for all $c'\in I\setminus \{c\}$.  Such an $\alpha$ is unique.

Suppose $I'\supset I$ is nef-minimal, and let $\beta =c+\sum_{c_i \in I'\setminus\{c\}} a_i c_i$ be a nef curve generated by $I'$.  By adding multiples of curves in $I\setminus\{c\}$, may assume $\mathbf{N}(c').\beta = 0$ for all $c'\in I\setminus\{c\}$.  By unicity of $\alpha$, $\beta -(c+\alpha)$ is a curve generated by $I'\setminus \{c\}$.  Therefore, $I'' := I_0\cup I'\setminus \{c\}$ generates the nef curve $\beta$ as well.  We may assume $I_0\subset \mathbf{I}(\mathbf{F}(I))$, as otherwise we may discard elements from $I''$ and retain a set generating $c+\alpha$.  Hence, $\mathbf{F}(I'')\supseteq \mathbf{F}(I')$, which proves the claim.  The condition stated in the lemma ensures that we do not add $c$ back to $I''$ in a cyclic fashion upon iteration of this criteria.


Criteria (3): Let $D$ and $I$ be as described, and suppose $I'\supset I$ is nef-minimal.  We may assume $I'$ does not contain a curve sweeping out $D$, as otherwise the criteria is automatic.   We need to show the existence of some $I''\nsupseteq I$, generating a nef curve, with $\mathbf{F}(I')\subseteq\mathbf{F}(I'')$.  Pick $c\in \mathbf{N}(D)\cap\mathbf{I}(\mathbf{F}(I))$, which is nonempty because $D\in\mathbf{N}(\mathbf{I}(\mathbf{F}(I)))$.  Clearly $\mathbf{F}(I')=\mathbf{F}(I'\cup\{c\})$, and $I'\cup\{c\}$ generates a nef curve but has been eliminated by one of the prior criteria applied to $I\cup\{c\}$.  If $I\cup\{c\}$ is eliminated by criteria (1), 
so is $I$ because $\mathbf{F}(I)=\mathbf{F}(I\cup\{c\})$.  Otherwise, criteria (2) applied to $c'\in I$ eliminates $I\cup\{c\}$.  
Thus, we may find a nef-generating $I''\not\ni c'$ with $\mathbf{F}(I'') \supseteq \mathbf{F}(I' \cup \{c\}) = \mathbf{F}(I')$.  If, for every choice of $c=c_i\in \mathbf{N}(D)\cap\mathbf{I}(\mathbf{F}(I))$, $I\cup\{c_i\}$ is not eliminated by criteria (1) 
but by criteria (2) for some $c_i'\in I$, then to avoid adding $c_i'$ back to $I''$ in a cyclic fashion, we must ensure the graph $G$ mentioned above contains no directed cycles after the addition of an edge from $c_i'$ to $c_i$ for some choice of $c_i$.  This is because, rather than eliminating $I\cup \{c_i\}$ with criteria (2), we effectively eliminate $I$.
\end{proof}

Lemmas \ref{propertiesNefMinimal} and \ref{eliminate} outline a procedure for efficiently finding nef-minimal subsets $I\subset \mathcal{C}$ not covered by a collection of nef curves $S$.  This serves as the basis for our algorithm below. 

\begin{alg}[\textit{Proving $\mathfrak{E} = \text{Eff}(X)$ using Negative Curves}]\label{negCurvesAlg}
Given a collection of nef curves $S$, we first check that it covers all nef-minimal $I\subset \mathcal{C}$.  This follows if we may eliminate every nef-minimal $I$ relative $S$.  To determine this, begin with a one element subset $I=\{c\}\subset\mathcal{C}$, and iteratively extend $I$ by curves $c'$ such that $\mathbf{N}(c')$ pairs positively with the most recently added curve.  Graphically, we extend the vertex set of $\mathbf{G}(I)$ one vertex at a time until a directed cycle appears.  
\begin{figure}[ht]
\includegraphics[width=0.7\linewidth]{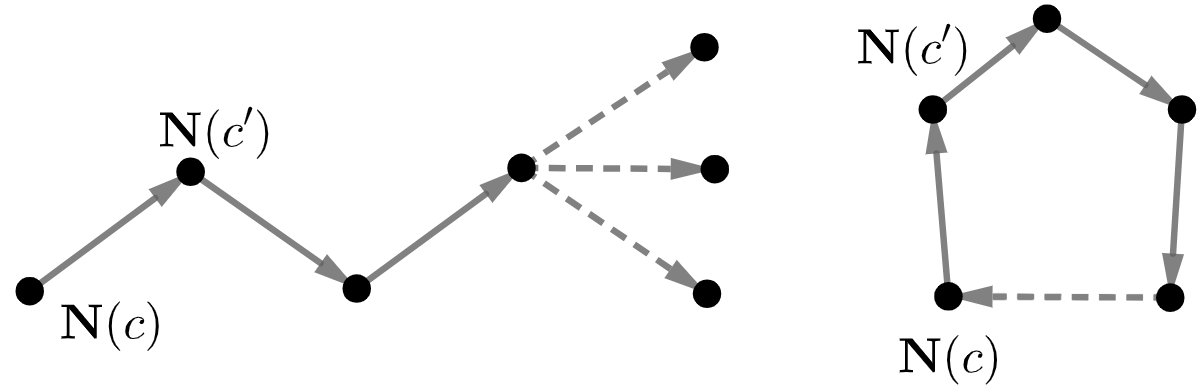}
\caption{Extending $I$ until a direct cycle appears in $\mathbf{G}(I)$.}
\end{figure}
If $I$ still does not generate a nef curve, add a curve $c'$ with $\mathbf{N}(c')\in\mathbf{P}(I)$, and continue.  We may stop extending $I$ if it has been eliminated.  Doing this in all possible ways, for all possible one element subsets of $\mathcal{C}$, will generate all uneliminated nef-minimal subsets of $\mathcal{C}$.

Suppose a collection of nef curves $S$ covers all nef-minimal $I\subset \mathcal{C}$.  To determine whether $\mathfrak{M}\subset \mathfrak{E}$ (and thus that $\mathfrak{E} = \text{Eff}(X)$), by Proposition \ref{nefCovers} we must show that for each $c\in S$, $\mathbf{F}(c)\subset \mathfrak{E}$.  Given $F=\mathbf{F}(c)$ for $c\in S$, let $L$ be the linear subspace of $N^1(X)$ defined by the vanishing of all curves in $\mathbf{I}(F)$.  
\begin{figure}[ht]
\includegraphics[width=0.23\linewidth]{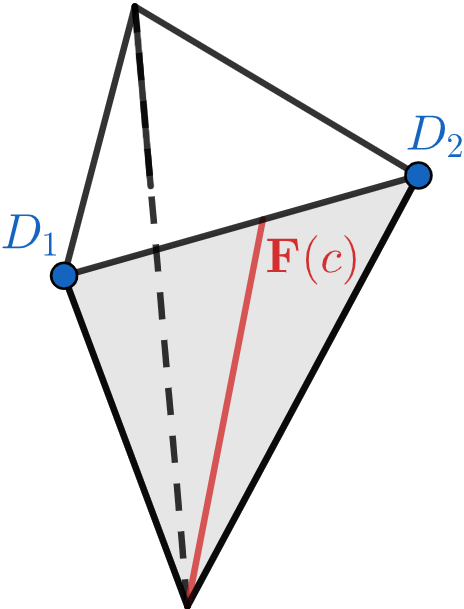}
\caption{An example when $\rho(X)=3$, with $\mathfrak{E}$ outlined in black, $\mathbf{F}(c)$ outlined in red, $D_1$, $D_2$ drawn in blue, and $\mathfrak{E}\cap L$ shaded darker.  
$\mathbf{F}(c)$ always has codimension at least 2 in $N^1(X)$.} 
\end{figure}
To show that $F\subseteq\mathfrak{E}$, it suffices to find elements $D_1,\dots, D_n\in \mathfrak{E}\cap L$ and show that for arbitrary $D\in F$, if $D-\epsilon D_i \not\in F$ for all $i$ and all $\epsilon > 0$, then $D\in \mathfrak{E}$.  Of course, symmetry may reduce the number of cases necessary to consider, as happens below.
\end{alg}

\begin{ex}\label{M05}
\textit{Computing $\text{Eff}(\overline{M}_{0,5})$}: Recall $\overline{M}_{0,5}\cong \text{Bl}_{p_0, p_1, p_2, p_3}\mathbb{P}^2$ has boundary divisors $E_i$ and $H - E_i - E_j$ for $i,j\in \{0,1,2,3\}$.  Each boundary divisor is an exceptional curve, and a divisor class $$D = dH - \sum m_i E_i$$ is effective upon restriction to them if and only if it pairs nonnegatively with them.  To show boundary divisors generate $\text{Eff}(\overline{M}_{0,5})$, let $\mathcal{C}=\mathcal{D}$ be the set of all boundary divisors, and define $\mathfrak{E}$ and $\mathfrak{M}$ as before.  Then for any $D\in \mathfrak{M}$,
\begin{itemize}
\item $D.E_i \geq 0$
\item $D.(H - E_i -E_j) \geq 0$
\end{itemize}
By Proposition \ref{nefCovers}, it is sufficient to check that $D\in \mathfrak{E}$ when $D.c = 0$ for some nef curve $c$ 
generated by $\mathcal{C}$.  Two obvious nef curves are $H - E_i = (H - E_i - E_j) + E_j$ and $2H - \sum E_i = (H - E_0 - E_1) + (H - E_2 - E_3)$.  Figure \ref{M05pic} contains the graph $\mathbf{G}(I)$ of subsets $I$ generating these curves. 
\begin{figure}[ht]
\includegraphics[width=0.575\linewidth]{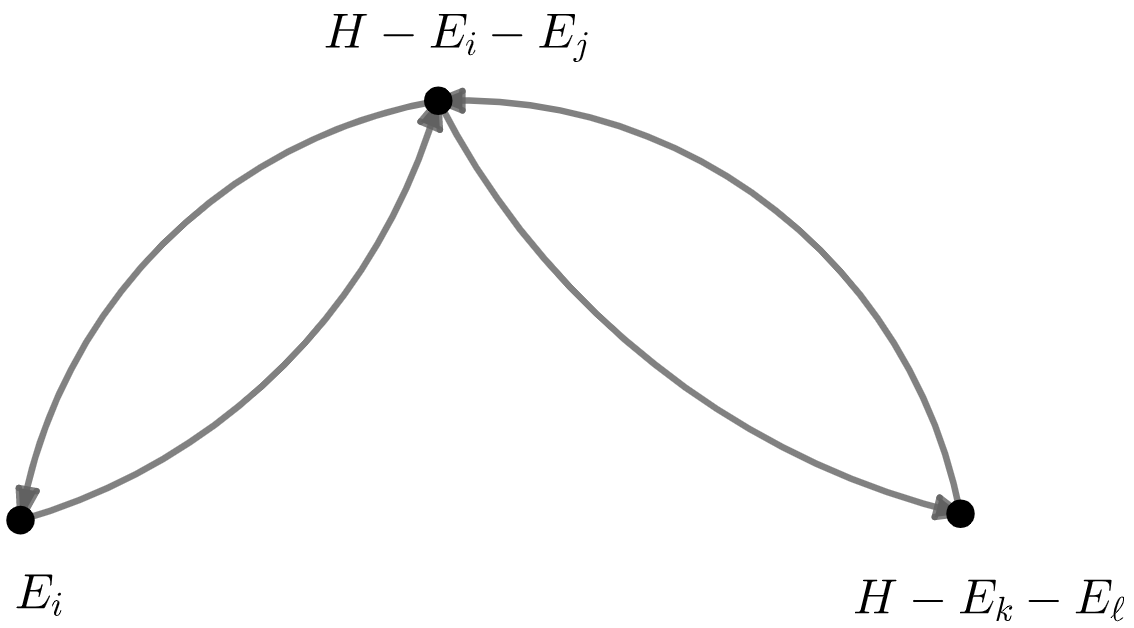}
\caption{The subset $I= \{E_i, H-E_i -E_j, H-E_k -E_\ell\}$ is not nef-minimal, but the two subsets $\{E_i, H-E_i -E_j\}, \{ H-E_i -E_j, H-E_k -E_\ell\}$ correlating the the cycles above are nef-minimal.}
\label{M05pic}
\end{figure}
Since $\mathcal{C}$ is unimodal (any two curves in $\mathcal{C}$ pair to either $0, 1, \text{ or } -1$), the nef-minimal subsets of $\mathcal{C}$ are precisely these pairs of curves.  In fact, the nef curves above are $S_5$ equivalent, so we may assume $D.(H -E_3) = 0$.  This implies $D.E_i = 0$ for $i\neq 3$, since $H - E_3 = (H - E_0 - E_3) + E_0  = (H - E_1 - E_3) + E_1  = (H - E_2 - E_3) + E_2$.  Thus,
$$D = dH - dE_3=d(H-E_0 -E_3) +d(E_0)$$
clearly belongs to $\mathfrak{E}$, and so $\mathfrak{E}=\text{Eff}(\overline{M}_{0,5})$.
\end{ex}

\begin{rem}
Algorithm \ref{negCurvesAlg} works particularly well for surfaces, as then each extreme divisor is its own negative curve, i.e. $|\mathbf{N}(D)| = 1$ for all $D\in \mathcal{D}$.
\end{rem}



\section{Proof of Theorem \ref{coneThm}}\label{Section10}

Using techniques developed in Section \ref{NegTech}, we prove the following statement: \textit{Let $n = 5,6$ and $\mathcal{B}$ be the set of boundary divisors on $\text{Eff}(\overline{M}_{0,n})$.  Suppose $D\in\text{Pic}(\overline{M}_{0,n})$ satisfies $D|_B \in \text{Eff}(B)$ for each $B\in\mathcal{B}$.  Then $D\in\text{Eff}(\overline{M}_{0,n})$.}

\begin{proof}[Proof of Theorem \ref{coneThm}]
The case $n=5$ is Example \ref{M05}.  For $n=6$, let $E_i$, $E_{ij}$, and $\Delta_{ijk}$ be the boundary divisors on $\overline{M}_{0,6}$ for $1\leq i < j < k \leq 5$. Let $KV_{ij,kh}= 2H - \sum_{\alpha} E_{\alpha} - E_{ik} -E_{ih} -E_{jk} -E_{jh}$ be the specified Keel-Vermiere divisor.  Let $e_i$ be the class of a general line in $E_i$, $e_{ij}$ the fiber over a general point in the blow-up of $l_{ij}$, and $l$ be class of the pullback of a general line in $\mathbb{P}^3$ under the Kapranov morphism.  Castravet \cite{castravet2008cox} shows the following: 
\begin{enumerate}
\item $e_i - e_{ij}$ and $2e_i - \sum_{j\neq i} e_{ij}$ are negative curve classes on $E_i$.
\item $l - e_i - e_{jk}$, $2l - e_i -e_j -e_k -e_{hl}$, and $l-  e_{ij}  -e_{ik} -e_{jk}  -e_{hl}$ are negative curve classes on $\Delta_{ijk}$.
\item $e_{ij}$ and $l -e_i -e_j +e_{ij}$ are negative curve classes on $E_{ij}$.
\end{enumerate}
A divisor $D\in \text{Pic}(\overline{M}_{0,6})$ is effective upon restriction to each boundary divisor if and only if $D$ pairs nonnegatively with each curve class listed above.  Writing 
$$D= dH -\sum_{1\leq i\leq 5} m_i E_i - \sum_{1\leq i < j \leq 5} m_{ij} E_{ij},$$
this translates to the following:
\begin{enumerate}
\item $m_i \geq m_{ij}$ for all $j\neq i$ and $2m_i \geq \sum_{j\neq i} m_{ij}$.
\item $d \geq m_i + m_{jk}$ (allowing roles of $i,j,k$ to vary), $2d \geq m_i + m_j +m_k + m_{hl}$, and $d\geq m_{ij} + m_{ik} + m_{jk} + m_{hl}$, where $\{i,j,k,h,l\} = \{1,2,3,4,5\}$
\item $m_{ij}\geq 0$ and $d+ m_{ij} \geq m_i +m_j$
\end{enumerate}
Let $\mathcal{C}$ be the collection of negative curves on boundary divisors listed above and $\mathcal{D}$ be the set of boundary and Keel-Vermiere divisors.  Identifying the minimal subsets of $\mathcal{C}$ that generate nef curves is more involved than it was for $\overline{M}_{0,5}$.  We relegate the proof of the following fact to a lemma:
\begin{lem}\label{lastLem}
There are $3$ $S_6$-equivalence classes of nef curves that cover all nef-minimal subsets of $\mathcal{C}$ on $\overline{M}_{0,6}$.  They are
\begin{enumerate}
\item $l -e_1$ ($S_6$ equivalent to $3l - \sum e_i$)
\item $l-e_{12} -e_{34}$ ($S_6$ equivalent to $2l - e_1 -e_2 -e_{34} -e_{35}$)
\item $2l -e_{12} -e_{13} -e_{14} -e_{25} -e_{35} -e_{45}$ ($S_6$ equivalent to $3l - 2e_5 - \sum_{i,j\leq 4} e_{ij}$)
\end{enumerate}
\end{lem}
\noindent A geometric interpretation of these extreme nef curves appears in \cite[Section~5.2]{Hassett_2002}.  By Proposition \ref{nefCovers}, it is sufficient to check that $\mathbf{F}(c)\subseteq\mathfrak{E}$ for each curve $c$ listed above, where $\mathfrak{E}$ is the cone generated by $\mathcal{D}$.  Let $D\in \mathbf{F}(c)$.  By $S_6$ symmetry, we reduce to 3 cases.  

\textbf{Case 1:} Suppose $c= l-e_1$.  Then $D.(l-e_1 -e_{ij})\geq 0$ and $D.e_{ij}\geq 0$ imply $m_{ij} = 0$ for $i,j\neq 1$.  Furthermore, $D.(l-e_1 -e_i + e_{1i})\geq 0$ and $D.(e_i -e_{1i})\geq 0$ imply $m_i = m_{1i}$ for $i\neq 1$.  We want to show that $D$ is the pullback of an effective divisor class from $\overline{M}_{0,5}$.  To see this, note $D.e_i \geq 0$ and $D.(l-e_i -e_j)= D.(l-e_i -e_j +e_{ij}) \geq 0$ for $i,j\neq 1$.  Thus, by the above proof for $\overline{M}_{0,5}$, $D$ is the pullback of an effective divisor class.

\textbf{Case 2:} Suppose $c=(l-e_{12} -e_{34})$.  Since $l-e_{12} -e_{34}$ can be written as $(l -e_i -e_{jk}) +(e_i -e_{ih})$, we see that $m_1=m_2 = m_{12}$, $m_3=m_4 =m_{34}$, and $m_5 \leq \min(m_{12}, m_{34})$.  Since $D.(l-e_{12} -e_{34} -e_{i5} -e_{j5}) = 0$, we see $m_{i5} = 0$ for all $i$.  Moreover, $D.(2e_1 - e_{12} -e_{13} -e_{14} -e_{15}) \geq 0$ implies that $m_{12} \geq m_{13} + m_{14}$.  Similarly, $m_{12} \geq m_{23} + m_{24}$, and $m_{34} \geq \max(m_{13} +m_{23} , m_{14} + m_{24})$.  

\textbf{Claim:}  for any positive linear combination of boundary and Keel-Vermiere divisors $D'$, as long as the pairing of $D-D'$ with the following curves satisfies the prescribed (in)equalities listed below,
\begin{enumerate}
\item $(D-D').e_{i5} = 0$
\item $(D-D').(e_i -e_{ij}) = 0$ for $\{i,j\}\in \{\{1,2\},\{3,4\}\}$
\item $(D-D').(l-e_{12} -e_{34}) = 0$
\item $(D-D').(l-e_5-e_{ij}) \geq 0$ for $(i,j)\in \{(1,2),(3,4)\}$
\item $(D-D').e_5 \geq 0$
\item $(D-D').(e_{ij} -e_{ih} -e_{ik}) \geq 0$ for $\{\{i,j\},\{h,k\}\}=\{\{1,2\},\{3,4\}\}$
\item $(D-D').e_{ij}\geq 0$ for $i,j\in\{1,2,3,4\}$
\end{enumerate}
then $(D-D')\in\mathbf{F}(l-e_{12} -e_{34})$ (pairs nonnegatively with all curves in $\mathcal{C}$, and to $0$ with $l-e_{12} -e_{34}$).  Note that $D$ already satisfies the (in)equalities listed above.  

To see the above claim, group curves in $\mathcal{C}$ by the divisor they sweep out.  Clearly all negative curves sweeping out any $E_{i5}$ or $E_i$ would pair nonnegatively with $(D-D')$.  For $l-e_i -e_j +e_{ij}$ with $(i,j)\in \{1,2\}\times\{3,4\}$, note that $(D-D').(l-e_i -e_j)=0$, so we only need to check $(D-D').e_{ij} \geq 0$.  For $\{i,j\}\in \{\{1,2\},\{3,4\}\}$, $(D-D').(e_i -e_{ij}) = 0$ and $(D-D').(l-e_{ij}) \geq 0$ implies $(D-D').(l-e_i -e_j +e_{ij})\geq 0$.  Lastly, $(D-D')$ would similarly pair nonnegatively with any curve class in $\mathcal{C}$ that sweeps out a hyperplane $\Delta_{ijk}$ (this is left to the reader).

The following divisors $D'$ are linear combinations of generators of $\mathfrak{E}$ (elements of $\mathcal{D}$) that satisfy all the \textit{equalities} listed above, and therefore belong to $L\cap\mathfrak{E}$ (where $L$ is the linear subspace of $N^1(\overline{M}_{0,6})$ given by the vanishing of curves in $\mathbf{I}(\mathbf{F}(l-e_{12}-e_{34}))$):
\begin{itemize}
\item $E_5$
\item $E_{13}, E_{14}, E_{23}, E_{24}$
\item $KV_{13,24}, KV_{14,23}$
\item $(E_2 +\Delta_{234}), (E_1+\Delta_{134}), (E_4 +\Delta_{124}), (E_3 +\Delta_{123})$
\item $(\Delta_{125} +E_{15} +E_{25}), (\Delta_{345} +E_{35} +E_{45})$
\end{itemize}

By Algorithm \ref{negCurvesAlg}, we assume $D-\epsilon D' \not\in\mathbf{F}(l-e_{12}-e_{34})$ for all $D'$ listed above, and show $D\in \mathfrak{E}$.

Since $E_5$ only pairs positively with (4), we may assume (by symmetry) $D.(l-e_5 -e_{12})=0$, and so $d=m_5 +m_{12}$.  Similarly, since $E_{13}, E_{14}, E_{23},\text{ and } E_{24}$ only pair positively with (6), which consists of $e_{12} -e_{13} -e_{14}$, $e_{12} -e_{23} -e_{24}$, $e_{34} -e_{13} -e_{23}$, and $e_{34} -e_{14} -e_{24}$, in order for all four $D'=E_{ij}$ to satisfy $D- \epsilon D' \not\in \mathbf{F}(l-e_{12} -e_{34})$, we'd need either $D.(e_{12} -e_{i3} -e_{i4}) = 0$ for $i=1,2$ or $D.(e_{34} -e_{1i} -e_{2i}) = 0$ for $i=3,4$.  However, the former option implies the latter holds as well, so we may assume $m_{34}= m_{13} + m_{23} = m_{14} + m_{24}$.

If $D.e_5 = 0$, then $0=D.(l-e_{12})=D.(l-e_1)$, so we return to case 1.  Similarly, $m_{12}= 0$ or $m_{34}= 0$ returns to case 1.  If $d= m_{34} + m_{5}$, then $m_{34} = m_{12} = m_{5} = d/2$.  But then as $ 2m_{12} \geq (m_{13} + m_{14}) + (m_{23} + m_{24}) = 2m_{34} = 2m_{12}$, this implies $m_{12} = m_{13} + m_{14} = m_{23} + m_{24}$.  So, since $m_{13} + m_{14} = m_{13} + m_{23} = d/2$, we see $m_{14} = m_{23}$ and $m_{13} = m_{24}$.  Thus, $D$ is exactly $m_{13} KV_{14,23} + m_{14} KV_{13,24}$. 

Thus, we can assume $m_5, m_{12},m_{34} > 0$, and $d > m_{34} + m_5$.  However, with $D' =(\Delta_{125} +E_{15} +E_{25})$, this implies $D-\epsilon D' \in \mathbf{F}(l-e_{12} -e_{34})$ unless $m_{12} = m_{13} + m_{14}$, so we may assume this as well.  But then, if $m_{12} = m_{23} + m_{24}$, we'd acquire $m_{12} = m_{34}$, giving $d= m_5 + m_{12} = m_5 + m_{34}$, contrary to our assumptions.  Hence, we may assume $m_{12} > m_{23} + m_{24}$.

If either $m_{13}$ or $m_{14}$ is 0, then $m_{12} \leq m_{34}$, which we just covered.  If $m_{23} = m_{24} = 0$, then $m_5 = m_{34} = m_{13} = m_{14}$, and $D= m_{34}(E_1+\Delta_{134}) + m_{12}(\Delta_{125} +E_{15} +E_{25}) +(m_{12}-m_{34})E_5$.  Thus, we may assume $m_{13},m_{14} > 0$ and one of $m_{23},m_{24} > 0$.  This, however, implies that $D- \epsilon D' \in \mathbf{F}(l-e_{12} -e_{34})$ for $D' = (E_4 +\Delta_{124})\text{ or } (E_3 +\Delta_{123})$, a contradiction.  Thus $D\in \mathfrak{E}$.

\textbf{Case 3:} Suppose $c=(2l - e_{12} - e_{13} -e_{14} -e_{25} - e_{35} -e_{45})$.  Then, as 
$$(2l - e_{12} - e_{13} -e_{14} -e_{25} - e_{35} -e_{45}) = 2(l -e_1 -e_5 + e_{15}) + (2e_1 - \sum e_{1i}) + (2e_5 - \sum e_{i5}),$$
we must have $D.(l- e_1 -e_5 + e_{15}) = D.(2e_1 - \sum e_{1i}) = D.(2e_5 - \sum e_{i5}) = 0$. Similarly, as 
$$(2l - e_{12} - e_{13} -e_{14} -e_{25} - e_{35} -e_{45}) = (l - e_{1i} - e_{1j} - e_{ij} - e_{k5}) +  (l - e_{i5} - e_{j5} - e_{ij} - e_{1k}) + 2e_{ij}$$
for $\{ i,j,k\} = \{2,3,4\}$, we must have
\begin{itemize}
\item $m_{ij}= 0$ for $i,j\in\{2,3,4\}$
\item $d - m_{ih} - m_{jh} -m_{ij} - m_{kl} = 0$ for $\{i,j,k\} = \{2,3,4\}$ and $\{h,l\} = \{1,5\}$
\item $2m_i - \sum m_{ij} = 0$ for $i = 1,5$
\item $d - m_1 - m_5 + m_{15} = 0$
\end{itemize}
From these conditions, we obtain $d - m_{12} - m_{13} - m_{45} = 0 = d - m_{25} - m_{45} - m_{13}$, so $m_{12} = m_{25}$.  Similarly, $m_{1i} = m_{i5}$ for $i=3,4$.  This implies $m_1 = m_5$.  Thus, we have the following:
\begin{itemize}
\item $m_{1i} = m_{i5}$ for $i\in \{2,3,4\}$.
\item $m_1 = m_5$
\item $d= m_1 + m_5 -m_{15}$
\item $2m_1 = \sum m_{1i}$
\item $m_{ij}=0$ for $i,j\in \{2,3,4\}$.
\item $d= m_{12} + m_{13} + m_{14}$
\end{itemize}
\textbf{Claim:} as long as $D$ satisfies these equalities and the following inequalities, then $D\in\mathbf{F}(2l - e_{12} - e_{13} -e_{14} -e_{25} - e_{35} -e_{45})$:
\begin{itemize}
\item $d-m_i -m_j \geq 0$ for $i,j\in\{2,3,4\}$
\item $m_i \geq m_{1i}$ for $i\in \{2,3,4\}$
\item $d\geq m_1 + m_{i5}$ for $i\in\{2,3,4\}$
\item $m_{1i} \geq 0$
\item $2d\geq m_1 + m_5  + m_i$ for $i \in \{2,3,4\}$.
\item $d - m_1 -m_i + m_{1i} \geq 0$ for $i\in \{ 2,3,4\}$
\item $2d\geq m_2 + m_3 + m_4 + m_{15}$
\end{itemize}
The proof of this is omitted, as it is similar to the argument in case 2. The following divisors $D'$ all satisfy the equalities listed above, and therefore belong to $L\cap\mathfrak{E}$ (where $L$ is the linear subspace of $N^1(\overline{M}_{0,6})$ given by the vanishing of curves in $\mathbf{I}(\mathbf{F}(2l - e_{12} - e_{13} -e_{14} -e_{25} - e_{35} -e_{45}))$):
\begin{itemize}
\item $ E_2, E_{3}, E_{4},\Delta_{125}, \Delta_{135}, \Delta_{145}$
\item $ (E_{1} + E_5 +  2E_{15}), (\Delta_{123} +\Delta_{235} + 2E_{23}), (\Delta_{124} +\Delta_{245} + 2E_{24}), (\Delta_{134} +\Delta_{345} + 2E_{34})$
\item $ KV_{15,23}, KV_{15,24}, KV_{15,34}$
\end{itemize}

Note, if $D.e_{1i} = 0$ for $i\in \{2,3,4\}$, then $0= m_{1i} = m_{i5}$, and this reduces to case 2.  
If $2d = m_1 + m_5 + m_i$ for some $i\in \{2,3,4\}$, say $i=2$, then we have $d= m_1 + m_2 -m_{12}$ (this sums with $d\geq m_5 +m_{12}$ to the get indicated equailty).  Since we also have $d=m_{12} + m_{13} + m_{14} = m_{12} + m_{35} + m_{45}$, we see that $2d = m_1 + m_2 + m_{35} + m_{45}$.  This is case 2.  If $D.(l-e_1 -e_i + e_{1i}) = 0$, then $D.(l-e_5 -e_i + e_{i5}) = 0$ as well, and we see (say $i=2$) that $4d = 2m_2 + m_1 + m_5 + m_{13} + m_{14} + m_{35} + m_{45}$, which reduces to case 2.  Lastly, if $2d = m_2 + m_3 + m_4 + m_{15}$, then $3d = \sum m_i$, which is case 1.  Thus, we may suppose these inequalities are strict, and keep track of only the following inequalities:  
\begin{itemize}
\item $d -m_i -m_j \geq 0$ for $i,j\in\{2,3,4\}$
\item $m_i \geq m_{1i}$ for $i\in \{2,3,4\}$
\item $d\geq m_1 + m_{i5}$ for $i\in\{2,3,4\}$
\item $m_{15} \geq 0$
\end{itemize} 

We may assume, for the divisors $D'$ written above, that $D-\epsilon D'\not\in \mathbf{F}(2l - e_{12} - e_{13} -e_{14} -e_{25} - e_{35} -e_{45})$ for all $\epsilon > 0$.  For $D'$ as specified, that implies the following inequalities are equalities:
\begin{itemize}
\item For $D' = E_i$, $d-m_i - m_j = 0$ for some $j\in \{2,3,4\}\setminus\{i\}$.  
\item For $D' = \Delta_{15i}$, either $m_{15} = 0$ or $d=m_j + m_k$ with $\{i,j,k\} = \{2,3,4\}$.
\item For $D' =  (E_{1} + E_5 +  2E_{15})$, $d = m_1 + m_{i5}$ for some $i\in \{2,3,4\}$
\item For $D' =  (\Delta_{1ij} +\Delta_{ij5} + 2E_{ij})$, either $d = m_1 + m_{k5}$, $m_i = m_{1i}$, or $m_j = m_{1j}$
\item For $D' = KV_{15,ij}$, either $d=m_1 + m_{k5}$ or $m_k = m_{1k}$.
\end{itemize}

Considering $D'=E_2, E_3, E_4$, we may therefore assume that $d=m_2 + m_3 = m_2 + m_4$.  Since action by $S_6$ makes $d= m_3 + m_4$ equivalent to $m_{15} = 0$, by $D' = \Delta_{125}$ we may also assume that $d=m_3 + m_4$.  This implies that $m_2 = m_3 = m_4 = d/2$.  

By $D'=KV_{15,ij}$, for each $k\in \{2,3,4\}$, either $d=m_1 +m_{k5}$ or $m_k = m_{k5}$.  If $m_k = m_{k5}$ for two distinct $k$, say $k=2,3$, then $d/2 = m_2 = m_3 = m_{25} = m_{35} = m_{12} = m_{13}$.  But this puts us in case 2, as $d= m_{12} + m_{35}$.  So, we may suppose $d = m_1 + m_{25}$, $d= m_1 + m_{35}$, $m_2 > m_{12}$, and $m_3 > m_{13}$.  With $D'=(\Delta_{123} +\Delta_{235} + 2E_{23})$, these last two inequalities imply $d= m_1 + m_{45}$.  So, we may assume $d= m_1 + m_{k5}$ for all $k\in \{2,3,4\}$.

This, however, specifies our divisor class completely.  We see $m_{12} = m_{13} = m_{14} = m_{25} = m_{35} = m_{45}$, so that $6m_{12} = 2d$.  Thus, $m_{12} = d/3$, and $m_1 =m_5 = 2d/3$.  This gives $m_{15} = 4d/3 -d = d/3$.  Thus, $6|d$, and $D$ is precisely $\frac{d}{3}\Delta_{125} + \frac{d}{6}(KV_{15,34} + \Delta_{135} + \Delta_{145} + E_1 + E_5 + 2E_{15})$.
\end{proof}

\section{Proof of Theorem \ref{rational contractions}}\label{Section9}



\begin{proof}
We recall the statement of Lemma \ref{lastLem}: 
There are $3$ $S_6$-equivalence classes of nef curves that cover all nef-minimal subsets of $\mathcal{C}$ on $\overline{M}_{0,6}$.  They are
\begin{enumerate}
\item $l -e_1$ ($S_6$ equivalent to $3l - \sum e_i$)
\item $l-e_{12} -e_{34}$ ($S_6$ equivalent to $2l - e_1 -e_2 -e_{34} -e_{35}$)
\item $2l -e_{12} -e_{13} -e_{14} -e_{25} -e_{35} -e_{45}$ ($S_6$ equivalent to $3l - 2e_5 - \sum_{i,j\leq 4} e_{ij}$)
\end{enumerate}
Interpreted using results in Section \ref{Section MDS}, these are precisely the extreme rays of $\text{Nef}_1(X)$ whose associated special contractions are initial.  In fact, there are rational contractions listed in \cite[Section~5.2]{Hassett_2002} satisfying the hypotheses of Proposition \ref{extreme contraction} for each such curve.  To obtain a special contraction associated to each ray, merely compute classes of boundary divisors contracted to points.  These are boundary divisors with negative curves meeting flopping curves on $\overline{M}_{0,6}$ which are contracted by the rational contractions $\phi : \overline{M}_{0,6} \dashrightarrow S$ identified in \cite[Section~5.2]{Hassett_2002}.  To conclude our proof, it suffices to prove Lemma \ref{lastLem}.
\end{proof}

\begin{rem}
Let $\phi = \pi_{12} \times \pi_{34}$ be the product of iterated forgetful morphisms.  Boundary divisors contracted to points by $\phi : \overline{M}_{0,6} \rightarrow \overline{M}_{0,4} \times \overline{M}_{0,4}$ are of class $E_{13}$, $E_{14}$, $E_{23}$, $E_{24}$, and $E_5$.  Writing $\overline{M}_{0,4} \times \overline{M}_{0,4} \cong \mathbb{P}^1 \times \mathbb{P}^1$, these lie over points $\{(0,0), (0,1), (1,0), (1,1), (\infty, \infty)\}$.  These points are the center of the blow-up $Y \rightarrow \overline{M}_{0,4} \times \overline{M}_{0,4}$.

Similarly, suppose $\phi : \overline{M}_{0,6} \dashrightarrow \mathbb{P}^2$ is the map described in \ref{rational contractions}(3) with general fiber of class $2l -e_{12} -e_{13} -e_{14} -e_{25} -e_{35} -e_{45}$.  We note $\phi$ is defined away from three 1-dimensional boundary strata of class $e_i - e_{1i} -e_{i5}$ for $i \in \{2,3,4\}$. the blow-up $Y \rightarrow \mathbb{P}^2$ has as center the 6 intersection points of four lines.  These four lines are the unique lines dominated by reducible divisors $D \subset \overline{M}_{0,6}$.  Fibers over their intersection points contain one of the strict transforms of boundary divisors $E_2$, $E_3$, $E_4$, $\Delta_{125}$, $\Delta_{135}$, and $\Delta_{145}$.  Our proof of Theorem \ref{coneThm} lists these divisors, along the reducible divisors $D$ and a few Keel-Vermiere divisors, as classes of divisors contracted by $\phi$.
\end{rem}

\subsection{Proof of Lemma \ref{lastLem}}\label{lastlemSection}
\begin{proof}
To find the minimal subsets of $\mathcal{C}$ that generate nef curves, it is important to use the action of $S_6$ on $\mathcal{C}$.  The curves $e_{ij}$ and $l -e_i -e_j +e_{ij}$ are $S_6$ equivalent.  Similarly, $\Delta_{ijk}$ is $S_6$ equivalent to $E_i$, and $2l-e_i -e_j -e_k -e_{hl}$ is $S_6$ equivalent to $l-e_i -e_{jk}$ and $e_i -e_{ij}$.  Lastly, $(2e_i -\sum e_{ij})$ is $S_6$ equivalent to $l-e_{12} -e_{13} -e_{23} -e_{45}$. Thus, there are exactly 3 $S_6$ orbits of curves in $\mathcal{C}$.  Of them, only $2e_i - \sum e_{ij}$ is not a \textit{unimodal} curve class (it pairs to $-2$ with $E_i$ but $1$ with $E_{ij}$).  We will refer to curve classes in this orbit as ``the $-2$ curve (on $D$)," curves sweeping out some $E_{ij}$ as ``edge curves," and curves equivalent to $e_1-e_{12}$ as ``$-1$ plane curves."  We refer to the 3 $S_6$-equivalence classes of curves in Lemma \ref{lastLem} as ``curve \ref{lastLem}(i)." We note that $\mathbf{P}(2e_1 - \sum e_{1i}) = \{E_{12}, E_{13}, E_{14}, E_{15}\}$, $\mathbf{P}(e_1 - e_{12}) = \{ E_{12}, \Delta_{134}, \Delta_{135}, \Delta_{145}\}$, $\mathbf{P}(e_{12}) = \{ \Delta_{12i}| i =3,4,5\}$.

In the language of Lemma \ref{eliminate}, we will show that we can eliminate the following subsets of $\mathcal{C}$:
\begin{enumerate}
\item $\{ e_i - e_{ij}, (l-e_i -e_j + e_{ij})\}$ ($S_6$ equivalent to $\{ (2l -e_1 -e_2 -e_3 -e_{45}), (l-e_4 -e_5 + e_{45})\}$ and $\{(l-e_1 -e_{23}), e_{23}\}$)
\item $\{(2e_i -\sum e_{ij}), (l-e_i -e_h +e_{ih}), (l-e_i -e_k +e_{ik})\}$ ($S_6$ equivalent to $\{ (l-e_{12} -e_{13} -e_{23} -e_{45}), e_{12}, e_{13}\}$ and $\{ (l-e_{12} -e_{13} -e_{23} -e_{45}), e_{12}, (l-e_4 -e_5 +e_{45})\}$)
\item $\{(l-e_i -e_{jk}), (e_i -e_{ih})\}$ ($S_6$ equivalent to $\{(2l -e_1 -e_2 -e_4 -e_{35}), (e_4 -e_{34})\}$ and $\{(l-e_1 -e_{34}), (l-e_2 -e_{35})\}$)
\item $\{(2e_i -\sum e_{ij}), (2e_k -\sum e_{jk}), (l-e_i -e_k +e_{ik})\}$ ($S_6$ equivalent to $\{ (l-e_{12} -e_{13} -e_{23} -e_{45}), (l-e_{12}-e_{14} -e_{24} -e{35}), e_{12}\}$ and $\{ (l-e_{12} -e_{13} -e_{23} -e_{45}), (2e_4 -\sum e_{i4}), (l-e_4 -e_5 +e_{45})\}$)
\item $\{ e_i - e_{ij}, e_{ij}\}$ ($S_6$ equivalent to $\{ (2l -e_1 -e_2 -e_3 -e_{45}), e_{45}\}$ and $\{(l-e_1 -e_{23}), (l-e_2 -e_3 +e_{23})\}$)
\item $\{ (2e_i -\sum e_{ij}), e_{ik}\}$ ($S_6$ equivalent to $\{ (l -e_{12} -e_{13} -e_{23} -e_{45}), l-e_1 -e_2 + e_{12}\}$ and $\{(l-e_{12} -e_{13} -e_{23} -e_{45}), e_{45}\}$)
\item $\{ (e_i -e_{ih}), (l-e_j -e_{ik})\}$ ($S_6$ equivalent to $\{ (l -e_1 -e_{23}), (l-e_2 -e_{45})\}$, $\{ (l-e_1 -e_{23}), (e_1 -e_{12})\}$, and $\{ (2l -e_1 -e_2 -e_3 -e_{45}), (e_1 -e_{12})\}$)
\end{enumerate}
We will refer to these as subsets 1-7.
After showing we can eliminate these subsets, the cycles that may appear in the graph $\mathbf{G}(I)$ of some nef-minimal, uneliminated $I\subset\mathcal{C}$ are quite limited, and a simple check finishes the proof.  Throughout the proof, we will use the following four relationships repeatedly (up to $S_6$ equivalence):
\begin{enumerate}
\item $(l-e_i -e_j +e_{ij}) + (e_i - e_{ik}) = (l-e_j -e_{ik}) + e_{ij}$ ($S_6$ equivalent to $e_{12} + (2l -e_1 -e_2 -e_3 -e_{45}) = (l -e_1 -e_2 +e_{12}) + (l-e_3 -e_{45})$)
\item $ (2e_j -\sum_{\alpha} e_{j\alpha}) + (l -e_i -e_j +e_{ij}) + (e_i -e_{ik}) = (l-e_{jh} -e_{jm} -e_{hm} -e_{ik}) + e_{hm} + (e_j -e_{ij})$
\item  $ (2e_j -\sum_{\alpha} e_{j\alpha}) + (l -e_i -e_j +e_{ij}) + e_{jk} = (e_j -e_{jm}) + (l-e_i -e_{jh})$ 
\item  $ (2e_j -\sum_{\alpha} e_{j\alpha}) + e_{ij} + e_{jk} = (e_j -e_{jm}) + (e_j -e_{jh})$
\end{enumerate}
We will refer to these as relationships 1-4.

The elimination of subsets (1)-(4) is trivial: 
they generate curves \ref{lastLem}(1), \ref{lastLem}(2), \ref{lastLem}(2), and \ref{lastLem}(3) respectively.  Eliminating subset (5) is also straightforward.  Let $I\supset\{ (e_i -e_{ij}), e_{ij}\}$ be a nef-minimal subset of $\mathcal{C}$.  Then, by Lemma \ref{propertiesNefMinimal}, $I$ must contain a curve $c$ with $\mathbf{N}(c).e_{ij} > 0$.  The only such $\mathbf{N}(c)$ are $\Delta_{ijk}$ for various $k$.  
\begin{figure}[ht]
\includegraphics[width=0.375\linewidth]{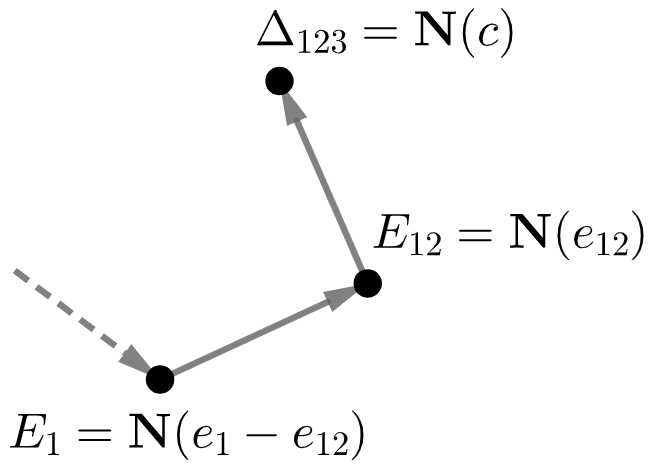}
\caption{Eliminating subset (5).}
\end{figure}
We may assume $(i,j,k)=(1,2,3)$, so that $c\in\{(l-e_1 -e_{23}),(l-e_{2} -e_{13}), (l-e_3 -e_{12}), (2l -e_1 -e_2 -e_3 -e_{45}), (l-e_{12} -e_{13} -e_{23} -e_{45})\}$.  Note that $e_1 -e_{12} + e_{12} = e_1 -e_{1\alpha} + e_{1\alpha}$ for all $\alpha$, so $(e_1 -e_{1\alpha}),e_{1\alpha} \in \mathbf{I}(\mathbf{F}(I))$ for all $\alpha$.  Thus, if $c=(l-e_1 -e_{23})$, $c+ (e_1 -e_{14})$ is curve \ref{lastLem}(2).  If $c=(l-e_2 -e_{13})$, $c+ e_{13}$ is curve \ref{lastLem}(1).  If $c=(l-e_3 -e_{12})$, $c+ e_{12}$ is also curve \ref{lastLem}(1).  If $c=(2l -e_1 -e_2 -e_3 -e_{45})$, $c + (e_1 -e_{14})$ is curve \ref{lastLem}(2).  Lastly, if $c=(l -e_{12} -e_{13} -e_{23} -e_{45})$, then $c + e_{12} + e_{13}$ is curve \ref{lastLem}(2).  Therefore, $\mathbf{F}(I)\subseteq\mathbf{F}(c)$ for $c$ a nef curve in $S_6$-equivalence class \ref{lastLem}(1) or \ref{lastLem}(2).  This eliminates $\{ (e_i -e_{ij}), e_{ij}\}$ by definition.

Eliminating subset (6) is slightly trickier.  To do so, we must first eliminate two other, larger subsets containing it:
\begin{itemize}
\item $\{(2e_i -\sum e_{ij}), e_{ik}, (l-e_i -e_h +e_{ih})\}$ ($S_6$ equivalent to $\{ (l-e_{12} -e_{13} -e_{23} -e_{45}), (l -e_1 -e_2 + e_{12}), e_{13}\}$, $\{ (l-e_{12} -e_{13} -e_{23} -e_{45}), e_{12}, e_{45}\}$, and $\{ (l-e_{12} -e_{13} -e_{23} -e_{45}), (l -e_1 -e_2 + e_{12}), (l-e_4 -e_5 + e_{45})\}$)
\item $\{(2e_i -\sum e_{ij}), e_{ik}, e_{ih}\}$ ($S_6$ equivalent to $\{ (l-e_{12} -e_{13} -e_{23} -e_{45}), (l -e_1 -e_2 + e_{12}), (l -e_1 -e_3 +e_{13})\}$ and $\{ (l-e_{12} -e_{13} -e_{23} -e_{45}), (l-e_1 -e_2 + e_{12}), e_{45}\}$)
\end{itemize}
The elimination of these subsets comes from Lemma \ref{eliminate}(2) applied to relationships (3) and (4) above, respectively.\footnote{This is the only application of Lemma \ref{eliminate}(2) in our proof, and every edge it adds to the graph $G$ in Lemma \ref{eliminate}(2) emanates from a $-2$ curve and ends at a $-1$ plane curve.  Later, a solitary application of Lemma \ref{eliminate}(3) will add edges to $G$ that start at $-2$ curves and stop at edge curves.  Both applications of \ref{eliminate}(2) or \ref{eliminate}(3) decrease the number of $-2$ curves present in a nef-generating subset of $\mathcal{C}$;  neither application increases the number of subsets of type (6) present in a nef-generating subset of $\mathcal{C}$.}  
To eliminate subset (6), we show that given nef-minimal $I\subset\mathcal{C}$ containing a subset equivalent to (6), either $\mathbf{F}(I)\subseteq\mathbf{F}(c)$ for $c$ a nef curve in Lemma \ref{lastLem}, or we may produce a nef-generating $I'\subset\mathcal{C}$ with $\mathbf{F}(I)\subseteq\mathbf{F}(I')$ such that $I'$ has either fewer $-2$ curves than $I$ or the same such number but fewer subsets equivalent to (6).  By descending double induction, this eliminates subset (6).

Suppose $I\subset\mathcal{C}$ is nef-minimal, and $\{(2e_1 -\sum e_{1\alpha}),e_{12}\}\subset I$.  Lemma \ref{propertiesNefMinimal} implies $I$ contains a curve $c$ with $\mathbf{N}(c).e_{12} > 0$.  The only such $\mathbf{N}(c)$ are $\Delta_{12i}$.  We first suppose all such $c\in I$ are not $-2$ curves.  Picking one $c$, we may assume $\mathbf{N}(c)=\Delta_{123}$ and $c\in\{ (l-e_3 -e_{12}), (l-e_1 -e_{23}), (l-e_2 -e_{13}), (2l -e_1 -e_2 -e_3 -e_{45})\}$.  If $c=(l-e_3 -e_{12})$, then $c+e_{12}$ is curve \ref{lastLem}(1).  Otherwise, by relationship (1), we may ``replace" $e_{12}$ with $(l-e_1 -e_2 +e_{12})$ and another curve, by arguing as in the proof of Lemma \ref{eliminate}(2).  For example, if $c=(l -e_1 -e_{23})$, then $e_{12} + (l-e_1 -e_{23}) = (l -e_1 -e_2 +e_{12}) + (e_2 -e_{23})$ implies $I' = \{(l-e_1 -e_2 +e_{12}), (e_2 -e_{23})\}\cup I\setminus\{e_{12}\}$ also generates a nef curve.\footnote{It is important that we avoid applying Lemma \ref{eliminate}(2) directly because this would not work!  In the graph $G$, we would add directed edges from $e_{12}$ to $l-e_1 -e_2 + e_{12}$, and vice versa.  Instead, we argue that we may eliminate $I$ using a subset $I'$ containing fewer subsets of type (6).  This, of course, requires that by ``replacing" $e_{12}$ with $\{(l-e_1 -e_2 +e_{12}), (e_2 -e_{23})\}$, we do not create additional subsets of type (6).  The only such subsets we might create involve $(l-e_1 -e_2 +e_{12})$ and a $-2$ curve on some $\Delta_{12i}$.  These are exactly the curves we excluded at the paragraph's beginning.}  In conclusion, we may assume instead that $I$ contains a $-2$ curve $c$ with $\mathbf{N}(c).e_{12} > 0$.  By symmetry, we may suppose $c=(l-e_{12} -e_{13} -e_{23} -e_{45})$, so that $\{(2e_1 -\sum e_{1\alpha}),e_{12}, (l-e_{12} -e_{13} -e_{23} -e_{45})\}\subset I$. 

\begin{figure}[ht]
\includegraphics[width=0.65\linewidth]{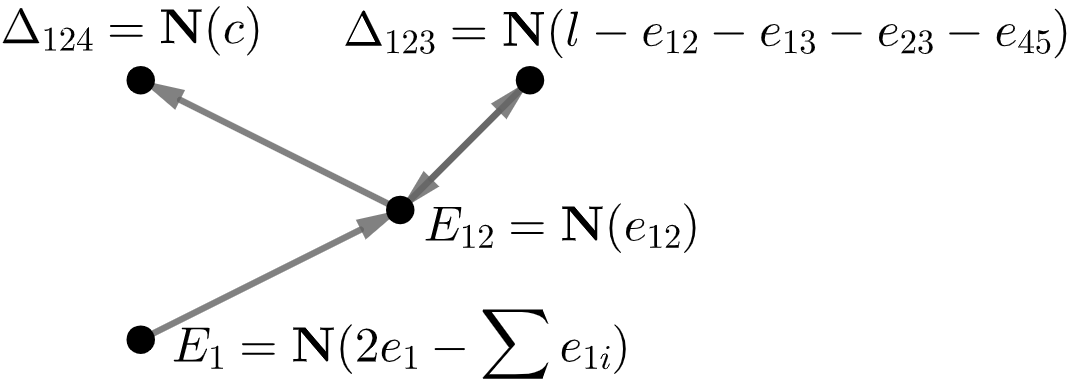}
\caption{Eliminating subset (6).}
\end{figure}
By Lemma \ref{propertiesNefMinimal}(2), we know that $I$ must contain another curve $c$ with $\mathbf{N}(c).e_{12} > 0$ or $\mathbf{N}(c).(l-e_{12} -e_{13} -e_{23} -e_{45}) > 0$.  However, the latter option implies $\mathbf{N}(c)\in\{E_{12}, E_{13}, E_{23}, E_{45}\}$, in which case $I\supset S$ for $S$ equivalent to subsets listed first or second above.  Thus, we must have $\mathbf{N}(c).e_{12} > 0$.  By symmetry, we may assume $\mathbf{N}(c)=\Delta_{124}$, and that $c\in\{(l-e_4 -e_{12}), (l-e_2 -e_{14}), (l -e_1 -e_{24}), (l -e_{12} -e_{14} -e_{24} -e_{35})\}$. If $c$ is the first or last curve class in this list, then $I$ generates curve \ref{lastLem}(1) or \ref{lastLem}(3).  If $c=(l-e_2-e_{14})$, then since 
$$(l-e_2 -e_{14}) + e_{12} + (l-e_{12} -e_{13} -e_{23} -e_{45}) = (l-e_{14} -e_{15} -e_{45} -e_{23}) +e_{15} + (l-e_2 -e_{13}),$$
we see that $e_{15} \in\mathbf{I}(\mathbf{F}(I))$.  By Lemma \ref{eliminate}(3), applied to $D= E_{15}$, we see that we may eliminate this option.\footnote{This is the solitary application of Lemma \ref{eliminate}(3).  It adds an edge to $G$ in Lemma \ref{eliminate}(2) starting at $(2e_1 - \sum e_{1\alpha})$ and ending at either $e_{12}$ or $e_{15}$.}  If $c=l-e_1 -e_{24}$, then since $(l-e_1 -e_{24}) +e_{12} = (l-e_1 -e_2 +e_{12}) + (e_2 -e_{24})$, we see $(l-e_1 -e_2 +e_{12}), (e_2 -e_{24})\in\mathbf{I}(\mathbf{F}(I))$.  Similarly, since $(l-e_1 -e_2 +e_{12}) + (e_2 -e_{24}) +(2e_1 -\sum e_{1i}) = (l-e_{13} -e_{15} -e_{35} -e_{24}) + e_{35} + (e_1 -e_{14})$, and
$ (l-e_1 -e_{24}) + (e_1 -e_{14}) = (l-e_{14} -e_{24} -e_{12} -e_{35}) +e_{12} +e_{35},$
we see $\{e_{12}, (l-e_{12} -e_{13} -e_{23} -e_{45}), (l-e_{12} -e_{14} -e_{24} -e_{35})\}\subseteq \mathbf{I}(\mathbf{F}(I))$.  Therefore $\mathbf{I}(\mathbf{F}(I))$ generates curve \ref{lastLem}(3).  This proves we can eliminate $\{ (2e_i -\sum e_{ij}), e_{ik}\}$.  

As a consequence of eliminating (6), note that if $I\subset \mathcal{C}$ is any uneliminated, nef-minimal subset that contains a $-2$ curve $c$, and $c' \in \mathbf{I}(\mathbf{F}(I))$ is an edge curve with $c.\mathbf{N}(c') = c' . \mathbf{N}(c) = 1$, then $c' \in I$.  For instance, let $c = (2e_1 - \sum e_{1\alpha}), c' = (l -e_1 -e_2 + e_{12})$, and suppose $c'\not\in I$. By Lemma \ref{propertiesNefMinimal}(2), $I$ must contain a curve $c''\neq c'$ with $\mathbf{N}(c'').c > 0$.  The only such $c''$ are edge curves, and if $\mathbf{N}(c).c'' = 0$ (for example $c'' = e_{12}$), then $\{c,c''\}$ is a subset of type (6).  Thus, up to symmetry $c'' = l - e_1 -e_3 -e_{13}$, and $c + c' + c''$ is curve \ref{lastLem}(2).  Thus $I$ would be eliminated.

Next, we eliminate $S=\{ (e_i -e_{ih}), (l-e_j -e_{ik})\}$.  
Suppose $I\subseteq \mathcal{C}$ is nef-minimal, and that $\{(e_1 -e_{12}),(l-e_3 -e_{14})\}\subset I$.  Note that by relationship (3), $\{(l-e_3 -e_{12}), (e_1 -e_{14}), (2e_1 -\sum_{\alpha} e_{1\alpha}), (l -e_1 -e_3 + e_{13}), e_{15}\}\subset\mathbf{I}(\mathbf{F}(I))$.  Lemma \ref{propertiesNefMinimal} implies $I$ contains a curve $c$ with $\mathbf{N}(c).(l-e_3 -e_{14}) > 0$.  By symmetry, we may assume $\mathbf{N}(c) = \Delta_{125}\text{ or } E_3$.  We address these in separate cases:
\begin{figure}[ht]
\includegraphics[width=0.5\linewidth]{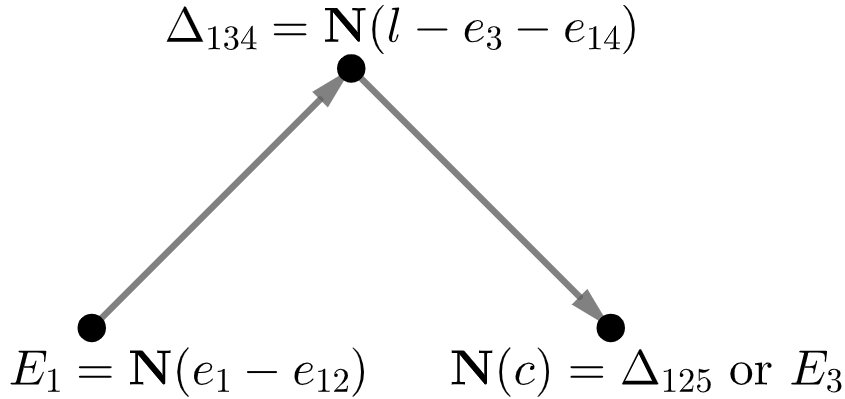}
\caption{Eliminating subset (7).}
\end{figure}

\textbf{Case 1:} Suppose $\mathbf{N}(c) = \Delta_{125}$.  Then $c\in \{ (l -e_1 -e_{25}), (l- e_2 -e_{15}), (l-e_5 -e_{12}), (2l -e_1 -e_2 -e_5 -e_{34}), (l-e_{12} -e_{15} -e_{25} -e_{34})\}$.  If $c=l -e_1 -e_{25}$, then since $(e_1 -e_{14})\in \mathbf{I}(\mathbf{F}(I))$, we see $I$ is covered by $l-e_{14} -e_{25}$.  A similar argument shows $c= (2l -e_1 -e_2 -e_5 -e_{34})$ is covered by $2l -e_2 -e_5 -e_{14} -e_{34}$.  Similarly, $e_{15}\in \mathbf{I}(\mathbf{F}(I))$ implies $c= l-e_2 -e_{15}$ is covered by $l-e_2$.  If $c=l-e_5 -e_{12}$, then $c+(l-e_3 -e_{14})= 2l -e_3 -e_5 -e_{12} -e_{14}$ is curve \ref{lastLem}(2).  Lastly, if $c= (l-e_{12} -e_{15} -e_{25} -e_{34})$, then we must also have $e_{15}\in I$, as well as another curve $c'\in I$ with $\mathbf{N}(c').e_{15} >0$.  This limits $\mathbf{N}(c')$ to either $\Delta_{135}$ or $\Delta_{145}$.  $c'$ cannot be the $-2$ curve on either, as this would generate curve \ref{lastLem}(2).  There are 8 other possibilities for $c'$.  They are all eliminated in exactly the same way options for $c$ were, except for $c' = l -e_5 -e_{13}$, which is handled by noting $c' + (l -e_1 -e_3 +e_{13}) + (e_1 -e_{12})= 2l -e_3 -e_5 -e_{12}$ is covered by curve \ref{lastLem}(2).

\textbf{Case 2:} Suppose $\mathbf{N}(c) = E_3$.  Then $c= e_3 -e_{3i}$ or $c= 2e_3 -\sum e_{3\alpha}$.  If $c= e_3 -e_{3i}$, then $i\neq 1,4,5$ may be seen by looking at $(l-e_1 -e_3 + e_{13})$ and $(l-e_3 -e_{12})$.  By looking at $l-e_3 -e_{14}$, we also eliminate $i=2$.  If $c= 2e_3 -\sum e_{3\alpha}$, then $c$, $(2e_1 -\sum e_{1\alpha})$, and $l-e_1 -e_3 + e_{13}$ generate curve \ref{lastLem}(3).  


This eliminates subset 7.  To conclude the proof, we apply Algorithm \ref{negCurvesAlg} to show the nef curves in Lemma \ref{lastLem} cover all nef-minimal subsets of $\mathcal{C}$.  We begin with a one element subsets $I = \{ c\} \subset \mathcal{C}$, and extend $I$ (or its graph $\mathbf{G}(I)$) until $I$ is eliminated or $\mathbf{I}(\mathbf{F}(I))$ is nef-generating.  Figure \ref{Final Image} displays a visual for this process.  As mentioned at the beginning of Section \ref{lastlemSection}, there are only 3 choices of $c$ up to symmetry: $(2e_1 -\sum e_{1i})$, $(e_1 -e_{12})$, and $e_{12}$.  To extend $\mathbf{G}(I)$, we must add $c'\in \mathcal{C}$ to $I$ such that $c.\mathbf{N}(c') > 0$.  

Suppose first that $c = (e_1 - e_{12})$, which gives $\mathbf{N}(c')=E_{12}$ or $\mathbf{N}(c') = \Delta_{1ij}$ for some $i,j \neq 2$.  If $\mathbf{N}(c')=E_{12}$, then $\{c,c'\}$ is either subset (1) or subset (5), which is eliminated.  Otherwise, $\mathbf{N}(c') = \Delta_{1ij}$, and $c'$ cannot be a $-1$ plane curve, as then $\{c,c'\}$ would be either subset (3) or (7), which is eliminated.  Thus, for any $-1$ plane curve $c$, every directed edge in $\mathbf{G}(I)$ emanating from the vertex labeled $\mathbf{N}(c)$ must end at a vertex $\mathbf{N}(c')$ for $c'$ a $-2$ curve.  

Suppose instead that $c=e_{12}$.  Deviating slightly from our standard argument, Lemma \ref{propertiesNefMinimal}(2) shows any nef minimal $I\supset\{e_{12}\}$ must contain $c_0$ with $c_0 . \mathbf{N}(e_{12}) >0$.  The only such $c_0$ satisfy $\mathbf{N}(c_0) = E_1,E_2,\Delta_{12i}$, or $\Delta_{345}$.  By the preceding paragraph, $c_0$ cannot be a $-1$ plane curve, and so must be a $-2$ curve.  If $\mathbf{N}(c_0) = E_1, E_2$, or $\Delta_{345}$, then $\{c_0,c\}$ is subset (6), and $I$ is eliminated.  Otherwise, $c.\mathbf{N}(c_0) = c_0 . \mathbf{N}(c) = 1$.  Following Algorithm \ref{negCurvesAlg}, we extend $I \{c,c_0\}$ by $c'$ with $\mathbf{N}(c').c >0$ or $\mathbf{N}(c').c_0 >0$.  If $\mathbf{N}(c').c_0 > 0$, $I$ contains either subset (2) or (6) and is eliminated.  If $\mathbf{N}(c').c >0$ and $c'$ is a $-2$ curve, then $I$ contains subset (4) and is eliminated.  So, we may assume $\mathbf{N}(c').c >0$, i.e. $\mathbf{N}(c') = \Delta_{12i}$, and that $c'$ is a $-1$ plane curve.  We note that any uncovered nef-minimal $I\subset\mathcal{C}$ must contain a $-2$ curve, and begin with this choice for $c$.

Suppose $I$ contains $2e_1 -\sum e_{1i}$.  By the preceding paragraphs, up to symmetry we may suppose it also contains $(l-e_1 -e_2 + e_{12})$, $(e_2 -e_{23})$, and the -2 curve on either $\Delta_{124}$ or $\Delta_{245}$.  The $-2$ curve on $\Delta_{124}$ pairs positively with $E_{12}$, while $\Delta_{124} .(l-e_1 -e_2 +e_{12}) = 0$, which is subset (6) (eliminated).  So, we can also suppose $I$ contains $l-e_{24} -e_{25} -e_{45} -e_{13}$.  Note that $e_{45}$ and $(l-e_1 -e_{23})$ vanish on $\mathbf{F}(I)$, as  
\begin{align*}
    (2e_1 +\sum e_{1i}) + (l-e_1 -e_2 +e_{12}) + (e_2 -e_{23}) &= (l-e_{14} -e_{15} -e_{45} -e_{23}) + e_{45} + (e_1 -e_{13}),\\
 \text{and } (l-e_1 -e_2 +e_{12}) + (e_2 -e_{23}) &= (l-e_1 -e_{23}) + e_{12}.
\end{align*}
So we must have $e_{45}\in I$ as well.  $I$ must contain another curve $c$ with $\mathbf{N}(c). e_{45} > 0$, which limits $\mathbf{N}(c)$ to $\Delta_{345}$ or $\Delta_{145}$.  Since $c$ must be a $-1$ plane curve, up to symmetry, if $\mathbf{N}(c) = \Delta_{345}$, $c$ would be $(l-e_4 -e_{35})$ or $(2l -e_3 -e_4 -e_5 -e_{12})$.  The first option generates curve \ref{lastLem}(2) with $(l-e_1 -e_{23})$, and the second generates curve \ref{lastLem}(1) with $(l-e_1 -e_2 + e_{12})$.  So, we may assume $\mathbf{N}(c) =\Delta_{145}$.  
\begin{figure}[ht]
\includegraphics[width=0.9\linewidth]{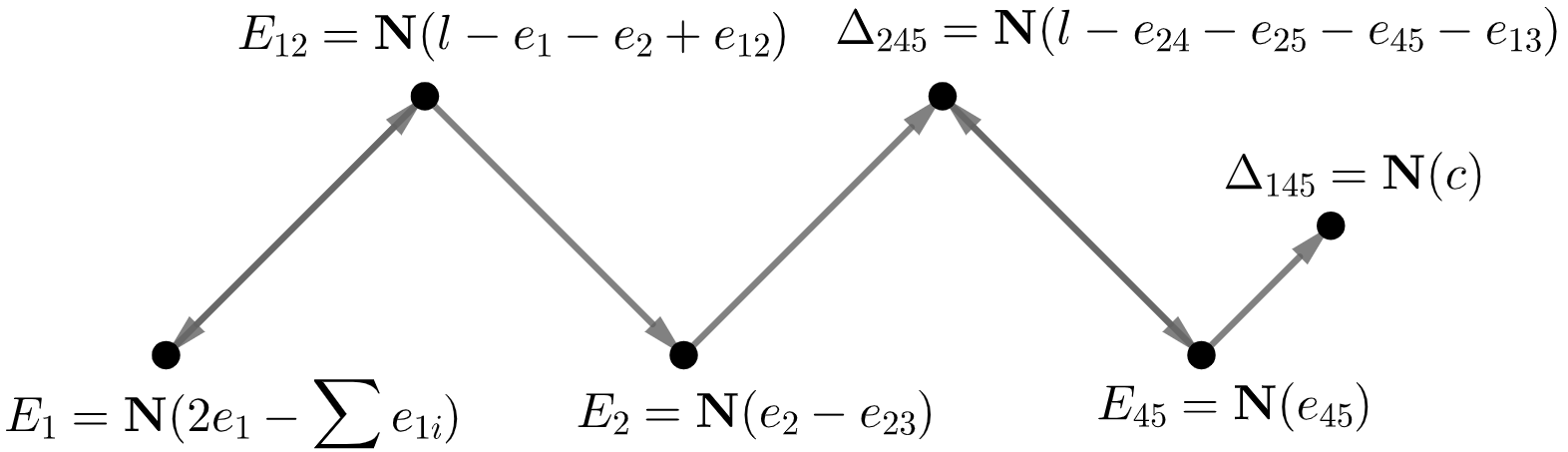}
\caption{Applying Algorithm \ref{negCurvesAlg}.}
\label{Final Image}
\end{figure}
Up to symmetry, there are 2 choices for $c$: $(l-e_5 -e_{14})$ and $(2l -e_1 -e_4 -e_5 -e_{23})$.  The second curve generates curve \ref{lastLem}(2) along with $e_1 -e_{13}$.  But, so too do we have
$$(l-e_5 -e_{14}) + (l-e_{24} -e_{25} -e_{45} -e_{13}) +e_{45} = (l-e_{13} -e_{14} -e_{34} -e_{25}) + e_{34} + (l-e_5 -e_{24}).$$
Along with $(l-e_1 -e_{23})$, $(l-e_5 -e_{24})$ generates curve \ref{lastLem}(2).  So, curves \ref{lastLem}(1)-(3) cover all nef-minimal $I$.
\end{proof}


\printbibliography

\end{document}